\documentclass[a4paper,10pt]{article}

\RequirePackage{amsthm,amsmath,amsfonts,amssymb}
\RequirePackage{mathrsfs,mathtools} 
\RequirePackage[authoryear]{natbib}
\RequirePackage{graphicx}
\RequirePackage{setspace}

\usepackage{geometry}
\usepackage{enumerate,enumitem}
\RequirePackage[colorlinks,citecolor=blue,urlcolor=blue]{hyperref}

\theoremstyle{plain}
\newtheorem{theorem}{Theorem}[section]
\newtheorem{lemma}[theorem]{Lemma}
\newtheorem{prop}[theorem]{Proposition}
\newtheorem{corollary}[theorem]{Corollary}

\theoremstyle{definition}

\newtheorem{remark}{Remark}
{\theoremstyle{plain}\newtheorem{assumption}{Assumption}}
{\theoremstyle{plain}}

\newcommand{\mc}[1]{\mathcal{#1}}

\newcommand{\contig}{\triangleleft}
\newcommand{\cov}{{\rm Cov}}
\newcommand{\darg}[1]{\,\mathrm{d}#1}
\newcommand{\dd}{{\rm d}}
\newcommand{\define}{\coloneqq}
\newcommand{\dmu}{\,\mathrm{d}\mu}
\newcommand{\E}{\mathbb{E}}
\newcommand{\effscrarg}[1]{\tilde{\ell}_{#1}}
\newcommand{\EP}{\mathbb{P}_n}
\newcommand{\eps}{\varepsilon}

\newcommand{\fracrootn}{\frac{1}{\sqrt{n}}}
\newcommand{\G}{\mathbb{G}_n}
\newcommand{\given}{\,|\,}
\newcommand{\half}{\tfrac{1}{2}}
\newcommand{\IP}[2]{\left\langle {#1}\, ,\, {#2} \right\rangle}
\newcommand{\limn}{\lim_{n\to\infty}} 
\newcommand{\meanin}{\frac{1}{n}\sumin}
\newcommand{\N}{\mathbb{N}}
\newcommand{\norm}[1]{\|#1\|}
\newcommand{\normal}{{\rm N}}
\newcommand{\PP}{\mathbb{P}}
\newcommand{\pl}{{\rm pl}}
\newcommand{\real}{\mathbb{R}}
\newcommand{\R}{\mathbb{R}} 
\newcommand{\sub}{{\rm sub}}
\newcommand{\sumin}{\sum_{i=1}^n}
\newcommand{\tr}{{\rm t}}
\newcommand{\weakconv}{\rightsquigarrow}

\newcommand\blfootnote[1]{%
  \begingroup
  \renewcommand\thefootnote{}\footnote{#1}%
  \addtocounter{footnote}{-1}%
  \endgroup
}

\begin{document}

\title{\sc \Large Semiparametrics via parametrics and contiguity
}

\author{{Adam Lee\thanks{BI Norwegian Business School, \href{mailto:adam.lee@bi.no}{adam.lee@bi.no}}, Emil A. Stoltenberg\thanks{University of Oslo, \href{emilas@math.uio.no}{emilas@math.uio.no}} and Per A. Mykland\thanks{University of Chicago, \href{mykland@uchicago.edu}{mykland@uchicago.edu}} }
}


\maketitle
\thispagestyle{empty}

\begin{abstract}
\onehalfspacing
\noindent

Inference on the parametric part of a semiparametric model is no trivial task. If one approximates the infinite dimensional part of the semiparametric model by a parametric function, one obtains a parametric model that is in some sense close to the semiparametric model and inference may proceed by the method of maximum likelihood. Under regularity conditions, the ensuing maximum likelihood estimator is asymptotically normal and efficient in the approximating parametric model. Thus one obtains a sequence of asymptotically normal and efficient estimators in a sequence of growing parametric models that approximate the semiparametric model and, intuitively, the limiting {`}semiparametric{'} estimator should be asymptotically normal and efficient as well. In this paper we make this intuition rigorous: we move much of the semiparametric analysis back into classical parametric terrain, and then translate our parametric results back to the semiparametric world by way of contiguity. Our approach departs from the conventional sieve literature by being more specific about the approximating parametric models, by working not only {\it with} but also {\it under} these when treating the parametric models, and by taking full advantage of the mutual contiguity that we require between the parametric and semiparametric models. 
We illustrate our theory with two canonical examples of semiparametric models, namely the partially linear regression model and the Cox regression model. An upshot of our theory is a new, relatively simple, and rather parametric proof of the efficiency of the Cox partial likelihood estimator.

\noindent 

\bigskip \noindent \textit{MSC2020 subject classifications}: Primary 62E20, 62G05; secondary 62G20.


\end{abstract}

\blfootnote{Adam Lee would like to thank participants at various conferences
for their helpful comments. Emil A. Stoltenberg would like to the Department of Data Science
at BI Norwegian Business School where the greater part of this work was carried out. Per
A. Mykland would like to thank the (U.S.) National Science Foundation for financial support
under grant DMS-2413952.
}

\clearpage

\section{Introduction}\label{sec::intro}
Drawing valid inference about the parametric component of a semiparametric model is often a challenging task. 
To deal with some of these challenges, we explore a simplifying strategy (inspired by research in high-frequency econometrics, see \citet{mykland2009inference}) in which the semiparametric problem is made (almost) fully parametric. The idea is to pretend that the data stem from a parametric distribution, allowing one to derive estimators and study their properties using standard parametric techniques. We let the these parametric distributions grow in such a way that they are mutually contiguous with respect to the full semiparametric distribution; and, finally, we use contiguity and Le Cam{'}s third lemma to switch the analysis back to the semiparametric model in the limit. 

The idea of approximating semiparametric models with growing parametric models is not new, as there is a well developed literature on sieves (\citet{grenander1981abstract}; see \citet{shen1997methods} and \citet{chenshen98} for seminal contributions to this literature, and \citet{chen07survey} for an excellent review). Our framework parallels the conventional sieve literature in its use of growing parametric models, but differs by carrying out the analysis {\it under} these parametric models themselves, i.e., assuming that the data stem from a member of such a parametric model. 
As we shall see, combining this analysis under parametric models with the use of contiguity to return to the semiparametric setting makes our approach -- which we call \emph{contiguous sieves} -- markedly different from that of \emph{conventional sieves} (i.e., growing parametric models, analysed under the same semiparametric distribution at all steps).
In the examples we have considered this approach allows us to establish the asymptotic normality and efficiency of semiparametric maximum likelihood estimators under conditions which are relatively user-friendly and straightforward to verify, essentially because they require just a little more than standard parametric likelihood theory. In particular, we are able to show asymptotic efficiency of semiparametric maximum likelihood estimators in a broad class of models without imposing any empirical process type conditions. 
Working with contiguous parametric submodels has several advantages. There is no ambiguity in defining a likelihood function (assuming the model is dominated), and based on this likelihood function, estimators can be derived or defined without difficulty. Carrying out the analysis {\it as if} these submodels were the models generating the data, enables us to work as though we have a correctly specified parametric model at each step; consequently, the required regularity conditions for these estimators to be asymptotically normal and efficient can be verified similarly as in classical parametric maximum likelihood theory. 
Finally, since the parametric distributions are constructed such that they are mutually contiguous to the true semiparametric distribution, we can transfer the analysis back to the semiparametric distribution in the limit, thus obtaining asymptotic results under the true distribution.

The article proceeds as follows. In Section~\ref{sec:setup} we outline the general setup we work in and motivate our approach. Section~\ref{sec:assumps} introduces the assumptions we work under. Sections~\ref{sec:asymp-eff} \&~\ref{sec:sieve-pl} contain our theoretical results. In Section~\ref{sec:applications} we provide a detailed analysis of two examples to demonstrate the application of our approach: the partially linear (regression) model and the Cox model. As we exemplify with the Cox model, parametric approximations might also be used as a purely theoretical tool (and not for estimation); this allows for efficiency proofs that we believe to be simpler than those currently available. Section~\ref{sec:concl} concludes.

\section{General setup and parametric theory}\label{sec:setup}
Let $X_1, \ldots, X_n$ be i.i.d.~replicates of $X$, where $X$ has distribution $P_0 \coloneqq P_{\theta_0, \eta_0}$ belonging to a semiparametric family of distributions $\mc{P}= \{P_{\theta, \eta}: \theta\in \Theta, \eta\in \mc{H}\}$, where $\Theta\subset \real^p$ for some $p \geq 1$, and $\mc{H}$ is a space of infinite dimension. We suppose that $\mc{P}$ is dominated by some $\sigma$-finite measure $\mu$, and write $p_{\theta,\eta}$ for the densities. The problem is to do inference on $\theta$ in the presence of the infinite dimensional nuisance parameter $\eta$. The simplifying strategy discussed in the introduction involves the construction of certain parametric submodels: For each $m \geq 1$, let $\mc{H}_m$ be a family of parametric functions $\eta_{\gamma}$ indexed by a parameter $\gamma \in \Gamma_m \subset \real^{k_m}$.
Typically these shall be such that $\mc{H}_m \subset \mc{H}_{m + 1}$ for all $m$ with $\cup_{m \geq 1} \mc{H}_m$ dense in $\mc{H}$ in an appropriate topology. 
Let $T_m \colon \mc{H}_m \to \Gamma_m$ be an isomorphism between the function space $\mc{H}_m$ and $\Gamma^{m}$ so that $T_m \eta_{\gamma} = \gamma \in \Gamma_m$ and $T_{m}^{-1}\gamma = \eta_\gamma \in \mc{H}_m$. Then for each $m \geq 1$,  
\begin{equation}
\mc{P}_m = \{	P_{\theta,T_m^{-1}\gamma} \colon \theta \in \Theta , \gamma \in \Gamma_m\},
\notag
\end{equation}
is a $p + k_m$ dimensional parametric model, and $P_{\theta,T_m^{-1}\gamma}$ has density $p_{\theta, T_m^{-1}\gamma}$ with respect to $\mu$. As discussed above, the idea is to carry out the analysis {\it as if} $X_1,\ldots,X_n$ was an i.i.d.~sample from a member of $\mc{P}_m$, let $m$ increase with the sample size $n$, and use contiguity to switch the analysis back to $P_{\theta_0,\eta_0}$ in the limit. The natural choice for a parametric approximation is to suppose that the sample stems from $P_{m} \coloneqq P_{\theta_0, T_m^{-1}\gamma_{0}}$, where $\theta_0$ equals the true value in the big model $P_0$, and $(\gamma_{0})_{m \geq 1} = (\gamma_{0,m})_{m \geq 1}$ is a sequence of growing vectors so that $\eta_{\gamma_{0}}$ approaches the true $\eta_0$ as $m$ tends to infinity. To declutter the notation, we avoid indexing ($\gamma$ and) $\gamma_0$ by $m$, as the size of these parameter vectors should be clear from the context. The densities of $P_0$ and $P_m$ with respect to $\mu$ are denoted $p_0$ and $p_m$, respectively, with similar subscripts for the expectations, $\E_0\, g(X)  = \int g(x)\,\dd P_0(x)$ and $\E_m\, g(X)  = \int g(x)\,\dd P_m(x)$. The product measure arising from an i.i.d.~sample of size $n$ is indicated by a superscript $n$, e.g.: $P_{\theta,\eta}^n = P_{\theta,\eta} \times \cdots \times P_{\theta,\eta}$. The score functions with respect to $\theta$ and $\gamma$ under $\mc{P}_m$ are
\begin{equation}
\dot{\ell}_{\theta,T_m^{-1}\gamma} = \frac{\partial }{\partial \theta} \log p_{\theta,T_m^{-1}\gamma},\quad \text{and}\quad 
\dot{v}_{\theta,T_m^{-1}\gamma} = \frac{\partial }{\partial \gamma} \log p_{\theta,T_m^{-1}\gamma}.
\notag
\end{equation} 
When evaluated in $(\theta_0,\gamma_0)$, i.e., the `true values' under $\mathcal{P}_m$, we write $\dot{\ell}_m = \dot{\ell}_{\theta_0,T_m^{-1}\gamma_0}$ and $\dot{v}_m = \dot{v}_{\theta_0,T_m^{-1}\gamma_0}$.
The log-likelihood function under $\mc{P}_m$ is $(\theta,\gamma) \mapsto \sum_{i=1}^n \log p_{\theta,T_m^{-1}\gamma}(X_i)$ for $(\theta,\gamma) \in \Theta \times \Gamma_m$, and the corresponding maximum likelihood estimator (MLE) for $\theta$ is $\widehat{\theta}_{m,n}$. Let $i_m$ be the Fisher information matrix under $P_m$ and block partition it as follows
\begin{equation}
i_m = 
\begin{pmatrix}
i_{m,00} & i_{m,01}\\
i_{m,10} & i_{m,11}\\
\end{pmatrix}
= \E_m\, 
\begin{pmatrix}
\dot{\ell}_m\dot{\ell}_m^{\tr} & \dot{\ell}_m\dot{v}_m^{\tr}\\
\dot{v}_m\dot{\ell}_m^{\tr} & \dot{v}_m\dot{v}_m^{\tr}\\
\end{pmatrix}.
\notag
\end{equation}
The standard maximum likelihood theory for inference on $\theta$ in the presence of a finite dimensional (fixed $m$) nuisance parameter $\gamma \in \Gamma_m$ goes as follows: Under regularity conditions (see, e.g., \citet[Theorem 5.39, p.~65]{vandervaart1998asymptotic}), maximum likelihood estimators are asymptotically linear in the parametric efficient influence function, that is, for fixed $m$, 
\begin{equation}
\sqrt{n}(\widehat\theta_{m, n} - \theta_0) = \frac{1}{\sqrt{n}}\sum_{i=1}^n J_m^{-1}\tilde{\ell}_m(X_i) + o_{P_{m}}(1),
\label{eq:parametric-asymp-linear}
\end{equation} 
as $n$ tends to infinity, where $\tilde{\ell}_m$ is the efficient score function and $J_m$ its variance:
\begin{equation}
\tilde{\ell}_m = \dot{\ell}_m - i   _{m,01}i_{m,11}^{-1} \dot{v}_m = \dot{\ell}_m - \Pi_{m}\dot{\ell}_m,\quad\text{and}\quad J_m =
i_{m,00} - i_{m,01}i_{m,11}^{-1}i_{m,10},
\notag
\end{equation}
with $\Pi_m$ the orthogonal projection onto the linear span of $\{\dot{v}_{m, j}: j=1, \ldots, k_m\}$ in $L_2(P_m)$. 

Provided the model $\mc{P}_m$ is differentiable in quadratic mean at $(\theta_0,\gamma_0)$ and the efficient information matrix $J_m$ is nonsingular, then~\eqref{eq:parametric-asymp-linear} is equivalent to $\widehat\theta_{m, n}$ being the best regular estimator~\citep[e.g.,][Lemma~8.14]{vandervaart1998asymptotic}. This means that as $n$ tends to infinity, but $m$ remains fixed, $\sqrt{n}(\widehat\theta_{n, m} - \theta_0)$ converges in distribution under $P_{m}^n$ to a mean zero normal distribution with variance matrix $J_m^{-1}$, this being the smallest possible asymptotic variance matrix of any regular estimator (see Section~\ref{subsec:true_semipara} for a formal definition). 

To illustrate our general strategy as well as the definitions and the notation introduced above, consider the partially linear regression model where $X = (W,Y,Z)$ for a real valued outcome $Y$ that given covariates $(W,Z) = (w,z)$ follows the regression model
\begin{equation}
Y = \eta(z) + \theta w + \eps. 
\label{eq:plm1}
\end{equation}
Here $\eps$ is a noise term independent of $(W,Z)$, $\eta$ is an infinite dimensional nuisance parameter, and $\theta \in \real$ is the parameter on which we seek to make inference. If $\eps$ is assumed to be mean zero normal with variance $\sigma^2$ and the covariates have a density, then the observation $(W,Y,Z)$ has a density. This density, however, cannot be used to define a maximum likelihood estimator for $(\theta,\eta)$, because the maximiser for $\eta$ will just interpolate the data (see \citet[pp.~221--226]{andersen1993statistical} for a discussion of these difficulties). 
To overcome these issues we instead pretend that $Y$ given $(W,Z) = (w,z)$ stems from the parametric model 
\begin{equation}
Y = \beta_m(z)^{\tr}\gamma + \theta w + \eps,
\label{eq:plm2}
\end{equation}
where $\beta_m = (\beta_{m,1},\ldots,\beta_{m,k_m})^{\tr}$ is a collection of  orthonormal (or other basis) functions, $\gamma = (\gamma_1,\ldots,\gamma_{k_m})^{\tr}$ is a Euclidean parameter vector, and $\eps$ and $(W,Z)$ have the same distribution as the similarly denoted random variables above. Now, maximum likelihood estimation is a least squares problem, and we readily obtain a maximum likelihood estimator for $(\theta,\gamma)$, say $(\widehat{\theta}_{m,n},\widehat{\gamma}_{m,n})$. Assuming that the data in fact stem from the parametric model in~\eqref{eq:plm2}, we get from standard parametric likelihood theory that $\sqrt{n}(\widehat{\theta}_{m,n} - \theta)$ converges in distribution to a mean zero normal with variance $J_m^{-1}$, where $J_m = \sigma^{-2}\big(\E\, W^2 - \sum_{j=1}^{k_m} (\E\, \{\beta_{m,j}(Z)W\})^2 \big)$, and that this is the efficient information under the model in~\eqref{eq:plm2}. 

For parametric inference ($m$ fixed) the conclusion above is in many ways the end of the maximum likelihood story: $\sqrt{n}(\widehat{\theta}_{m,n} - \theta_0) \weakconv \normal(0,J_m^{-1})$, and $J_m^{-1}$ is the smallest possible variance (of a regular estimator). Part of the motivation for the present paper, however, is the observation that if we let $m$ tend to infinity, it is often the case that $J_{m} \to J$ for some nonsingular matrix $J$. It is then tempting to conjecture both that (i) $\sqrt{n}(\widehat\theta_{m_n, n} - \theta_0) \weakconv \normal(0,J^{-1})$ under $P_{m_n}^n$ for some subsequence $(m_n)_{n \geq 1}$ tending to infinity with $n$; and (ii) that $J$, being the limit of a sequence of efficient information matrices, must be semiparametrically efficient under $\mc{P}$. Of course, we desire $\sqrt{n}(\widehat\theta_{m_n, n} - \theta_0) \weakconv \normal(0,J^{-1})$ under $P_0^n$ rather than $P_{m_n}^n$; we shall subsequently restrict the class of approximating models such that we can make this change of measure.

A case in point is the partially linear model where it follows from Parseval{'}s identity that with $J_m$ the efficient information under~\eqref{eq:plm2}
\begin{equation}
J_m \to \sigma^{-2} \E\,( W - (\E\, \{W \given Z\} )^2,
\notag
\end{equation}
where the limit is positive (provided $W$ is not a.s.~equal to $\E\,(W \given Z)$). Here we also recognise the limit as the efficient information under the semiparametric model in~\eqref{eq:plm1} (see, e.g., \citet[p.~110]{BKRW98}), demonstrating one case in which the limit of (parametric) efficient information matrices is efficient.     

\begin{remark} A well known {`}two-steps weak convergence{'} lemma (see, e.g., \citet[Theorem~4.2, p.~25]{billingsley1968} or \citet[Theorem 4.28, p.~77]{kallenberg2002foundations}) says that if $Z_{m,n}\weakconv Z_m$ for each $m$ and $Z_{m} \weakconv Z$ and there is a subsequence $(m_n)_{n \geq 1}$ such that $\lim_m\limsup_{n}\Pr(\|Z_{m,n} - Z_{m_n,n}\| \geq \eps) \to 0$ for any $\eps > 0$, then $Z_{m_n,n}\weakconv Z$. With $Z_{m,n} = \sqrt{n}(\widehat{\theta}_{m,n} - \theta_0)$, as in the setting outlined in the above paragraph, it is tempting to attempt to couple this two-steps theorem with the mutual contiguity $P_0^n \contig \triangleright\, P_{m_n}^n$ in order to conclude that $Z_{m_n,n}$ converges weakly to $\normal(0,J^{-1})$ under $P_0^n$. A closer look at the proof of the two-steps lemma, however, reveals that this conclusion would require $P_0^n$ to be contiguous with respect to $P_m^n$ {\it for any fixed $m$}. But since both $P_0^n = P_0 \times \cdots \times P_0$ and $P_m^n = P_m \times \cdots \times P_m$ with $P_0 \neq P_m$, $P_0^n$ cannot be contiguous with respect to $P_m^n$ (for this impossibility, see~\citet[Theorem 1]{oosterhoff1979note} or \citet[Lemma~V.1.6, p.~286]{jacod2003limit}). Whether the two-steps lemma can be coupled with contiguity appears to be an open problem.  \end{remark}

\section{Assumptions for contiguity and efficiency}\label{sec:assumps} 
For our subsequent efficiency results to make sense, we need to impose some structure on the semiparametric model. In Section~\ref{subsec:true_semipara} we outline this structure, introduce some notation, and define what we mean by asymptotic efficiency. 
In Section~\ref{subsec:para_approx} the conditions we impose on the parametric approximations are presented, along with a few lemmata easing the verification of these. 

\subsection{The true semiparametric model and efficiency}\label{subsec:true_semipara}
We concentrate on smooth models for i.i.d.~data, as in the classical parametric theory. This is made precise by imposing a differentiability in quadratic mean (DQM) condition on the semiparametric model $\mc{P}$ \citep{lecam1986, BKRW98}. Let $B$ be a linear space. We will consider measures $P_{\theta_0 + \tau/\sqrt{n}, \eta_n(b)}$ for $h = (\tau, b)\in \real^p \times B$ where $\eta_n(b)\to \eta_0$ as $n$ tends to infinity, and $\eta_n(0) = \eta_0$. 

\begin{assumption}[DQM]\label{ass:DQM}
For each $h\in \real^p \times B$,  $P_{\theta + \tau/\sqrt{n}, \eta_n(b)}\in \mc{P}$ for all large enough $n$, and $\limn \int \{\sqrt{n}(p_{\theta_0 + \tau/\sqrt{n}, \eta_n(b)}^{1/2}-p_0^{1/2}) - \half Ah \, p_0^{1/2}\}^2 \,\dd\mu = 0$, where $A$ is a bounded linear map between $\real^p \times B$ and $L_2(P_0)$.  
\end{assumption} 
It follows from the convergence in Assumption~\ref{ass:DQM} that $Ah\in L_2(P_0)$ and $\int Ah \,\dd P_0 = 0$ \citep[see e.g.][Lemma 25.14, p.~363]{vandervaart1998asymptotic}. Since $A$ is assumed to be linear we can split out the contributions of the parametric parameter of interest from the infinite dimensional nuisance as $Ah = \tau^\tr \dot{\ell} + Db$, where $\dot{\ell}$ is the ordinary score function for $\theta$ in a model where the nuisance $\eta$ is fixed; while $D \colon B \to \real$ is a linear operator, and $Db$ has the interpretation of a score function for $\eta$ with $\theta$ fixed. 

An estimator $\widehat{\theta}_n$ of $\theta_0$ is said to be \emph{regular} if $\sqrt{n}(\widehat{\theta}_n - \theta_0 - \tau/\sqrt{n}) \weakconv L$ under $P_{\theta_0+\tau/\sqrt{n}, \eta_n(b)}$, for some law $L$ and each $(\tau, b)\in \R^p\times B$. Requiring regularity excludes superefficient estimators such as the Hodges-Le Cam estimator (see e.g. \citet[Example~8.1, p.~109]{vandervaart1998asymptotic}). The efficiency bound for regular estimators for estimation of $\theta_0$ is determined by the efficient score,
\begin{equation}
    \tilde{\ell} = \dot{\ell} - \Pi\dot{\ell},
\notag
\end{equation}
where $\Pi$ denotes the orthogonal projection onto the closure of $\{Db: b\in B\}$  in $L_2(P_0)$. The efficiency bound is the inverse (provided it exists) of the variance of $\tilde{\ell}$, 
\begin{equation}
J = \E_0\, \tilde{\ell}(X)\tilde{\ell}(X)^\tr.
\notag
\end{equation}
More precisely, provided $J$ is nonsingular, by the H{\'a}jek--Le Cam convolution theorem, the limiting distribution of any regular sequence of estimators can be represented by the convolution of a $\normal(0, J^{-1})$ with some probability distribution (see e.g. \citealp[][Theorem 3.3.2]{BKRW98}; \citealp[][Theorem 25.20 \& Lemma 25.25]{vandervaart1998asymptotic}). As such, any regular estimator whose limiting  distribution is $L = \normal(0, J^{-1})$ is called \emph{best regular}; this is what we refer to as asymptotic efficiency.

The discussion immediately above required $J$ to be nonsingular. In fact, this is necessary for the existence of regular estimators~\citep[Theorem~2, p.~194]{chamberlain86asymptotic} and therefore we assume this throughout the rest of the paper.
\begin{assumption}[Nonsingularity]\label{ass:nonsingular-effinfo} $J$ is nonsingular.
\end{assumption}

\subsection{The parametric approximations}\label{subsec:para_approx}
We now introduce our assumption on the parametric approximations $P_m = P_{\theta_0,T_m^{-1}\gamma_0}$ to $P_0 = P_{\theta_0,\eta_0}$. In particular, we assume that $P_0$ can be approximated, in an appropriate sense, by a sequence of contiguous alternatives, similar to those in Assumption \ref{ass:DQM}. Since $P_0$ is unknown, checking Assumptions~\ref{ass:sieve-approx} and~\ref{ass:effscr-approx} below entails in practice checking it for any member of $\mc{P}$. Formally, however, these assumptions are required to hold only for the specific member $P_0$ of $\mc{P}$ that generated the data. 

\begin{assumption}[Contiguity]\label{ass:sieve-approx} There is a subsequence $(m_n)$ and a function $g$ such that $\E_0\,g(X) = 0$, $\E_0\,g(X)^2$ is finite, and $\E_0\,g(X) \tilde{\ell}(X) = 0$, and the log-likelihood ratio satisfies
\begin{equation}
        \log \frac{\dd P_{m_n}^n }{\dd P_{0}^n} = \fracrootn\sumin g(X_i) - \half\E_0\,g(X_1)^2  + o_{P_0^n}(1).
\label{rem:sieve-DQM-LAN-equivalence:eq:LAN}       
\end{equation}
\end{assumption}
In our applications, we will seek that $g=0$, but the orthogonality $\E_0\,g(X) \tilde{\ell}(X) = 0$ is all that is actually needed for our results. The log-likelihood ratio expansion in~\eqref{rem:sieve-DQM-LAN-equivalence:eq:LAN} is equivalent to the densities $p_{m_n} = \dd P_{m_n}/\dd \mu$ satisfying the DQM type condition
\begin{equation}
      \limn \int \{\sqrt{n}(p_{m_n}^{1/2} - p_0^{1/2}) - \half g \, p_0^{1/2}\}^2 \,\dd \mu = 0,
      \label{eq:sieve-DQM}
   \end{equation}
for $g$ such that $\E_0\,g(X) \tilde{\ell}(X) = 0$, that is, for the function $g$ appearing in~\eqref{rem:sieve-DQM-LAN-equivalence:eq:LAN}. See for example~\citet[Corollary 75.9, p.~386]{strasser1985mathematical} or~\citet[Prop.~17.2, p.~584]{lecam1986} for the equivalence of~\eqref{rem:sieve-DQM-LAN-equivalence:eq:LAN} and \eqref{eq:sieve-DQM}. Thus, to check Assumption~\ref{ass:sieve-approx} it suffices to show either 
~\eqref{rem:sieve-DQM-LAN-equivalence:eq:LAN} or \eqref{eq:sieve-DQM}.

The log-likelihood ratio expansion in~\eqref{rem:sieve-DQM-LAN-equivalence:eq:LAN} is key to our results. In particular, since the data are assumed i.i.d.,~Assumption~\ref{ass:sieve-approx} and the central limit theorem yield 
\begin{equation}
\log \frac{\dd P_{m_n}^n }{\dd P_{0}^n}
\overset{P_0^n}\weakconv Z, \quad\text{where}\quad Z \sim \normal\big(- \half \E_0\, g(X)^2,\E_0\, g(X)^2\big).
\label{eq:contiguityGaussianLimit}
\end{equation} 
Since $\exp(Z)$ is positive and $\E\exp(Z) = 1$, we get from Le Cam{'}s first lemma that $P_0^n$ and $P_{m_n}^n$ are mutually contiguous (see \citet[Example~6.5, p.~89]{vandervaart1988large}). Provided $\sqrt{n}(\widehat{\theta}_{m_n,n} - \theta_0)$ converges jointly with  $\log (\dd P_0^n /\dd P_{m_n}^n)$ to a Gaussian limit under $P_{m_n}^n$, the asymptotic distribution of $\sqrt{n}(\widehat{\theta}_{m_n,n} - \theta_0)$ under $P_{0}^n$ can be recovered by Le Cam{'}s third lemma \cite[][Prop.~6.7, p.~90]{lecam1986}. As such, Assumption~\ref{ass:sieve-approx} restricts the possible change in the limiting distribution of $\sqrt{n}(\widehat{\theta}_{m_n,n} - \theta_0)$ resulting from a change of measure from the parametric $P_{m_n}$ back to the semiparametric $P_0$. In particular, provided the limiting distribution of $\sqrt{n}(\widehat{\theta}_{m_n,n} - \theta_0)$ under $P_{m_n}^n$ is orthogonal to $g$ (i.e. $\limn \cov(\sqrt{n}(\widehat{\theta}_{m_n,n} - \theta_0),n^{-1/2}\sum_{i=1}^n g(X_i) ) = 0$), changing the measure from $P_{m_n}$ to $P_0$ does not affect the limiting distribution. 

In some examples, approximating models satisfying Assumption~\ref{ass:sieve-approx} can be derived directly from the submodels used to establish Assumption~\ref{ass:DQM}, given the similarity between the DQM required in Assumption~\ref{ass:DQM} and equation~\eqref{eq:sieve-DQM}. We view this as a virtue of our contiguous sieve framework which closely links the submodels used to define semiparametric efficiency with those used to estimate $\theta$. 

\begin{remark} As discussed above, the weak convergence in~\eqref{eq:contiguityGaussianLimit}, hence Assumption~\ref{ass:sieve-approx}, implies mutual contiguity of $ P_{m_n}^n$ and $P_{0}^n$. Assumption \ref{ass:sieve-approx} also imposes additional structure in that it requires the log-likelihood ratios to admit a local asymptotic normality (LAN) type expansion. This allows us ensure that no bias is incurred when switching back to the semiparametric law $P_0$ by imposing only the orthogonality condition $\E_0\,g(X) \tilde{\ell}(X) = 0$. In principle, this LAN-type expansion is not required for our overall strategy: one requires only the contiguity of $P_0^n$ to $P_{m_n}^n$ and the joint convergence of a sequence of statistics and the log-likelihood ratio under $P_{m_n}^n$ to apply the general form of Le Cam's third lemma (e.g., \cite[Prop.~6.5, p.~88]{lecam1986} or \cite[Theorem 6.6, p.~90]{vandervaart1998asymptotic}). With such a weaker requirement (as compared to Assumption \ref{ass:sieve-approx}), however, providing conditions under which no bias obtains when switching back from $P_{m_n}$ to $P_0$ becomes more complex. 
\end{remark}

In addition to the contiguity condition in Assumption~\ref{ass:sieve-approx}, we require that the efficient scores in the parametric submodels approximate the efficient score of the full semiparametric model in a statistically relevant sense. As we will take our limits along the subsequence $(m_n)_{n\geq 1}$ of Assumption~\ref{ass:sieve-approx}, it is sufficient that this approximation holds along this subsequence. This is convenient as Assumption~\ref{ass:sieve-approx} implies that $P_{m_n} \to P_0$ in total variation (via \eqref{eq:sieve-DQM}), which can help to simplify the demonstration of~\eqref{eq:effscr-approx} (see Lemma~\ref{lemma:effscr-approx-based-on-Pmn-TV} below).

\begin{assumption}[Efficient score approximation]\label{ass:effscr-approx} The efficient scores $\tilde{\ell}_{m_n}$ exist and 
    \begin{equation}\label{eq:effscr-approx}
        \limn \int \|\tilde{\ell}_{m_n} p_{m_n}^{1/2} -  \tilde{\ell}p_0^{1/2}\|^2 \,\dd\mu = 0,
    \end{equation}
where $\norm{x} = (\sum_{j=1}^p x_j^2)^{1/2}$ is the Euclidean distance.
\end{assumption}
Explicitly performing the orthogonal projection to compute $\tilde{\ell}$ can, in many models, be quite challenging. Fortunately, one may verify Assumption~\ref{ass:effscr-approx} without explicitly performing this projection, as the following lemma demonstrates. Here the space $B$ and the linear operator $D$ are as defined in connection with Assumption~\ref{ass:DQM}.
\begin{lemma}\label{lem:score-approx-implies-effscr-approx} Suppose that the scores $\dot{\ell}_{m_n}$ and $\dot{v}_{m_n}$ exist in the DQM sense. If
    \begin{equation}\label{eq:score-approx-parametric}
      \limn \int \|\dot{\ell}_{m_n}\,p_{m_n}^{1/2} -  \dot{\ell}p_0^{1/2}\|^2 \,\dd\mu = 0,
    \end{equation} 
    and for any $b\in B$ there are vectors $a_{m_n} \in \real^{k_{m_n}}$, such that 
    \begin{equation}\label{eq:score-approx-nonparametric}
    \limn \int  (a_{m_n}^\tr\dot{v}_{m_n}\,p_{m_n}^{1/2} -  Db p_0^{1/2})^2 \,\dd\mu = 0,
    \end{equation}
then Assumption~\ref{ass:effscr-approx} holds.    
\end{lemma}
\begin{proof} Under the assumption of the lemma, for large enough $n$, $\tilde{\ell}_{m_n}$ exists as soon as $\dot{\ell}_{m_n}$ and $\dot{v}_{m_n}$ do. Given \eqref{eq:score-approx-parametric} and \eqref{eq:score-approx-nonparametric}, apply Theorem~\ref{thm:projection-convergence} in the appendix to obtain \eqref{eq:effscr-approx}.
\end{proof}

We now present two lemmas which can help with the verification of~\eqref{eq:effscr-approx}, or of~\eqref{eq:score-approx-parametric} and \eqref{eq:score-approx-nonparametric}. Their straightforward proofs are deferred to Appendix~\ref{app:tech}. 

\begin{lemma}\label{lemma:effscr-approx-based-on-Pmn-TV} Suppose Assumption \ref{ass:sieve-approx} holds. If $f_{m_n} = f_0 + o_{P_0}(1)$ and $f_{m_n}^2$ is uniformly $P_{m_n}$-integrable then $\limn \int \|f_{m_n}p_{m_n}^{1/2} - f_0p_0^{1/2}\|^2\,\dd\mu = 0$.
\end{lemma}
That $f_{m_n}^2$ is uniformly $P_{m_n}$-integrable means that $\lim_{K\to\infty}\sup_n \int_{|f_{m_n}| \geq K} f_{m_n}^2\,\dd P_{m_n} = 0$. An alternative approach which circumvents the requirement to establish uniform square integrability directly can be based on a lemma originally due to \cite{riesz1928convergence}
(cf.~\citealp[Proposition~2.29, p.~22]{vandervaart1998asymptotic}). This lemma also connects Assumption~\ref{ass:effscr-approx} with the observation that in certain models the sequence of parametric efficient information matrices has the semiparametric efficient information matrix as a limit, as discussed in Section~\ref{sec:setup}. 
\begin{lemma}\label{lemma:Fisher_approx_implies_effscr_approx} Suppose that {\rm(i)} $p_{m_n}\to p_0$ in $\mu$-measure or {\rm(ii)} for any measurable set $A$, $P_0(A)\le \liminf_{n\to\infty} P_{m_n}(A)$. If $f_{m_n}\to f_0$ in $\mu$-measure and  $\limsup_{n\to\infty} \int f_{m_n}^2\,\dd P_{m_n} \le \int f_0^2\,\dd P_0<\infty$, then $\limn \int \|f_{m_n}p_{m_n}^{1/2} - f_0p_0^{1/2}\|^2\,\dd\mu = 0$.
\end{lemma} 
If we use either of these lemmas to verify~\eqref{eq:effscr-approx} directly, we are required to first find the efficient score under the semiparametric model $\mc{P}$, a task which involves performing sometimes complicated projections. Lemma~\ref{lem:score-approx-implies-effscr-approx}, on the other hand, does not involve any projections. The following corollary permits us to find the efficient semiparametric score $\tilde{\ell}$ as a limit. 
\begin{corollary}\label{corollary:to_lemma3.1} 
Suppose that~\eqref{eq:score-approx-parametric} and~\eqref{eq:score-approx-nonparametric} hold. If there is a vector $f_0$ of functions such that the conditions of either Lemma~\ref{lemma:effscr-approx-based-on-Pmn-TV} or~\ref{lemma:Fisher_approx_implies_effscr_approx} hold for $f_{m_n}$ with components $P_{m_n}$-a.s.~equal to each of the components of $\tilde{\ell}_{m_n}$, then $f_0 = \tilde{\ell}$ $P_0$-a.s..
\end{corollary}
\begin{proof} Since Lemma~\ref{lem:score-approx-implies-effscr-approx} holds, we have that $\int \norm{\tilde{\ell}_m p_m^{1/2} - \tilde{\ell}p_0^{1/2}}^2\,\dd\mu \to 0$, which is Assumption~\ref{ass:effscr-approx}. By either Lemma~\ref{lemma:effscr-approx-based-on-Pmn-TV} or~\ref{lemma:Fisher_approx_implies_effscr_approx} we have $\int \norm{\tilde{\ell}_m p_m^{1/2} - f_0p_0^{1/2}}^2\,\dd\mu \to 0$.
Since $L_2$ limits are unique up to sets of measure zero, $f_0p_{0}^{1/2} = \tilde{\ell}p_{0}^{1/2}$ $\mu$-almost surely, hence $f_0 = \tilde{\ell}$ $P_0$-almost surely.
\end{proof}

\section{Asymptotic efficiency}\label{sec:asymp-eff} We now present the main result of the paper. In order to approximate the semiparametric model, we let $m$ increase with $n$. This means that, for our purposes, the property corresponding to the asymptotic linearity in the parametric efficient score function exhibited in~\eqref{eq:parametric-asymp-linear}, is 
\begin{equation}
\sqrt{n}(\widehat\theta_{m_n,n} - \theta_0) = \frac{1}{\sqrt{n}}\sum_{i=1}^n J_{m_n}^{-1}\tilde{\ell}_{m_n}(X_i)+o_{P_{m_n}^n}(1).
\label{eq:sieve-asymp-linear}
\end{equation}
Note that property \eqref{eq:sieve-asymp-linear} is \emph{not} implied by \eqref{eq:parametric-asymp-linear}, as \eqref{eq:parametric-asymp-linear} only requires that the remainder $\sqrt{n}(\widehat\theta_{m,n} - \theta_0) -n^{-1/2}\sum_{i=1}^n J_{m}^{-1}\tilde{\ell}_{m}(X_i)$ is $o_{P_{m}^n}(1)$ as $n \to \infty$, for fixed $m$. Verification of~\eqref{eq:sieve-asymp-linear} depends on the definition of the estimator $\widehat\theta_{m_n,n}$. In Section~\ref{sec:sieve-pl} we use a profile likelihood technique to provide sufficient conditions for \eqref{eq:sieve-asymp-linear} to hold for a sequence of maximum likelihood estimators $\widehat{\theta}_{m_n,n}$ in growing contiguous parametric submodels.

Combined with Assumptions \ref{ass:DQM}--\ref{ass:effscr-approx}, the linearity of $\sqrt{n}(\widehat\theta_{m_n,n} - \theta_0)$ in the influence function displayed in~\eqref{eq:sieve-asymp-linear} implies the asymptotic efficiency of the estimator $\widehat\theta_{m_n, n}$. The essential idea for proving this is outlined in the following heuristic argument: If Assumption~\ref{ass:effscr-approx} holds, then $J_{m_n}^{-1}\tilde{\ell}_{m_n}$ in~\eqref{eq:sieve-asymp-linear} may be replaced by $J^{-1}\tilde{\ell}$, thus~\eqref{eq:sieve-asymp-linear} becomes $\sqrt{n}(\widehat\theta_{m_n,n} - \theta_0) = n^{-1/2}\sum_{i=1}^n J^{-1}\tilde{\ell}(X_i)+o_{P_{m_n}^n}(1)$. 
Combined with Assumption~\ref{ass:sieve-approx}, this asymptotic linearity ensures that $\sqrt{n}(\widehat\theta_{m_n,n} - \theta_0)$ converges jointly with $ \log (\dd P_0^n /\dd P_{m_n}^n)$ under $P_{m_n}^n$, and hence, by Le Cam's third lemma, one can change the measure from $P_{m_n}$ back to $P_0$ at the cost of adding a bias term of $J^{-1}\E_0\,\tilde{\ell}(X)g(X)$ to the limiting distribution of $\sqrt{n}(\widehat\theta_{m_n,n} - \theta_0)$ under $P_0$. But by the condition on $g$ in Assumption \ref{ass:sieve-approx}, this bias term is zero.
\begin{theorem}\label{thm:asymp-linear-efficiency} If Assumptions~\ref{ass:DQM}~{\&}~\ref{ass:nonsingular-effinfo} hold, and $(m_n)$ is a subsequence such that Assumptions~\ref{ass:sieve-approx}~{\&}~\ref{ass:effscr-approx} hold, and~\eqref{eq:sieve-asymp-linear} is satisfied, then $\widehat{\theta}_{m_n, n}$ is best regular in $\mc{P}$.
\end{theorem}
\begin{proof} Assumptions~\ref{ass:sieve-approx} and \ref{ass:effscr-approx} along with the i.i.d.~assumption on the data, verify the conditions of Proposition A.10 in~\citet[p.~185]{vandervaart1988large}. Applied to our setting, this proposition gives that $n^{-1/2}\sumin (\tilde{\ell}_{m_n}(X_i) - \tilde{\ell}(X_i))  
 = \sqrt{n}(\E_{m_n}\tilde{\ell}_{m_n}(X)
   - \E_0\,\tilde{\ell}(X)\,) - \E_0\, \tilde{\ell}(X) g(X) + o_{P_0^n}(1) = o_{P_0^n}(1)$.
The first equality follows from the cited proposition. The second equality ensues because $\E_{m_n}\tilde{\ell}_{m_n}(X) = 0$, $\E_0\,\tilde{\ell}(X) = 0$, and $\E_0\, \tilde{\ell}(X) g(X) = 0$ by Assumption~\ref{ass:sieve-approx}. Since Assumption~\ref{ass:sieve-approx} implies that $P_{m_n}^n$ and $P_0^n$ are mutually contiguous, we can swap the $o_{P_{m_n}^n}(1)$ in~\eqref{eq:sieve-asymp-linear} with $o_{P_0^n}(1)$, so that~\eqref{eq:sieve-asymp-linear} reads $\sqrt{n}(\widehat\theta_{m_n,n} - \theta_0) = n^{-1/2}\sum_{i=1}^n J_{m_n}^{-1}\tilde{\ell}_{m_n}(X_i)+o_{P_{0}^n}(1)$. Moreover, for any $a \in \real^p$ the reverse triangle inequality and then Cauchy--Schwarz yield
\begin{align*}
|(a^{\tr}J_{m_n}a)^{1/2} - (a^{\tr}Ja)^{1/2}|^2 
& = |\norm{a^{\tr}\tilde{\ell}_{m_n}p_{m_n}^{1/2} }_{\mu} - \norm{a^{\tr}\tilde{\ell}p_0^{1/2} }_{\mu}        |^2\\
& \leq \norm{a^{\tr}(\tilde{\ell}_{m_n}p_{m_n}^{1/2} -\tilde{\ell}p_0^{1/2} )     }_{\mu}^2
\leq \norm{a}^2 \int \|\tilde{\ell}_{m_n} p_{m_n}^{1/2} -  \tilde{\ell}p_0^{1/2}\|^2 \,\dd\mu,
\end{align*}
where the right hand side tends to zero by Assumption~\ref{ass:effscr-approx}. By continuity of the square function, this entails that $a^{\tr}J_{m_n}a \to a^{\tr}Ja$, and since the above is true for any $a \in \real^p$, $J_{m_n} \to J$
and therefore $J_{m_n}^{-1} \to J^{-1}$ since the inverse is a continuous operation when $J$ is nonsingular. Combining this with $n^{-1/2}\sum_{i=1}^n (\tilde{\ell}_{m_n}(X_i) - \tilde{\ell}(X_i)) = o_{P_0^n}(1)$ and~\eqref{eq:sieve-asymp-linear}, we conclude that
  \begin{equation}
    \sqrt{n}(\widehat{\theta}_{m_n,n} - \theta_0) = \frac{1}{\sqrt{n}}\sum_{i=1}^n J^{-1}\tilde{\ell}(X_i) + o_{P_0^n}(1).
    \label{eq:semiparam-asymp-linear}
 \end{equation}
Given Assumptions \ref{ass:DQM} and \ref{ass:nonsingular-effinfo}, the result then follows from Lemma~25.23 and Lemma~25.25 in \citet[pp.~367--369]{vandervaart1998asymptotic}.
\end{proof}

\section{Contiguous sieve MLE}\label{sec:sieve-pl}
In addition to Assumptions~\ref{ass:DQM}--\ref{ass:effscr-approx}, Theorem~\ref{thm:asymp-linear-efficiency} requires that the linear expansion \eqref{eq:sieve-asymp-linear} holds. In this section we provide two sets of conditions under which this linear expansion is satisfied for MLEs in growing contiguous parametric models. Both sets of conditions are growing parametric versions of a profile likelihood theorem due to \cite{murphy2000profile}. These authors provide conditions under which the semiparametric profile likelihood admits a quadratic expansion which, in turn, implies a condition like \eqref{eq:semiparam-asymp-linear}. In this section we argue similarly, but replace the {\it semiparametric} profile likelihood with a {\it contiguous sieve} profile likelihood which permits us to conclude that \eqref{eq:sieve-asymp-linear} holds. There is a key technical advantage of working with contiguous sieve profile likelihoods in place of the semiparametric profile likelihood. The latter requires the careful construction of {`}approximately least favourable submodels{'}, which can be quite complicated to construct, as can be seen from the examples in the cited article. In contrast, as the likelihoods we work with are parametric, exact least favourable submodels can be constructed following a clear recipe as they always take the same form.

To introduce the profile likelihood, let $(\theta,\gamma) \mapsto L_m(\theta,\gamma)(x) = p_{\theta,T_m^{-1}\gamma}(x)$ be the likelihood function under $\mc{P}_m$, and $L_{m,n}(\theta,\gamma) = \prod_{i=1}^n L_m(\theta,\gamma)(X_i)$ be the likelihood based on an i.i.d.~sample $X_1,\ldots,X_n$. Denote by $\pl_{m,n}(\theta)$ the profile likelihood based on $\mc{P}_m$
\begin{equation}
    \pl_{m,n}(\theta)
    = \sup_{\gamma \in \Gamma_m}L_{m,n}(\theta,\gamma),
    \notag
\end{equation}
    For each $\theta$, let $\widehat{\gamma}(\theta)$ be the value achieving this supremum, that is $\pl_{m,n}(\theta) = L_{m,n}(\theta,\widehat{\gamma}(\theta))$. We assume throughout that for large enough $n$, such a value exists.

\subsection{Quadratic expansion \& log concavity}
A straightforward set of sufficient conditions for \eqref{eq:sieve-asymp-linear} (or~\eqref{eq:semiparam-asymp-linear}) can be obtained using the results of~\cite{HjortPollard93}. Specifically, let $A_{m,n}(h) = \log \pl_{m,n}(\theta_0 + h/\sqrt{n}) - \log\pl_{m,n}(\theta_0)$, then, if the functions $h \mapsto A_{m,n}(h)$ are concave and one manages to find a subsequence $(m_n)$ such that 
\begin{equation}
A_{m_n,n}(h)
=  \frac{h^{\tr}}{\sqrt{n}}\sum_{i=1}^n \tilde{\ell}_{m_n}(X_i) - \half h^{\tr}J_{m_n}h + o_{P_{m_n}^n}(1),
\label{eq:Anh_LAN}
\end{equation}
for each $h$, then the {`}Basic Corollary{'} in \cite{HjortPollard93} immediately delivers \eqref{eq:sieve-asymp-linear}. This setting covers a large class of semiparametric models of practical interest, including the examples we study in detail in Section \ref{sec:applications} below. 
We emphasise that the concavity requirement is local, being imposed only on $A_{m, n}$ (actually only the $A_{m_n, n}$ being concave suffices). 

For cases in which $A_{m, n}$ is concave and ~\eqref{eq:Anh_LAN} can be shown to hold, this provides a complete proof of asymptotic normality and efficiency of the contiguous sieve MLE without requiring any empirical process type arguments. We summarise this in a proposition. 
\begin{prop}\label{theorem:sieve_sandwich} 
Suppose that Assumptions~\ref{ass:DQM}~\&~\ref{ass:nonsingular-effinfo} hold and that $(m_n)$ is a subsequence such that Assumptions~\ref{ass:sieve-approx}~{\&}~\ref{ass:effscr-approx} and the quadratic expansion in~\eqref{eq:Anh_LAN} hold. If $h \mapsto A_{m_n,n}(h)$ is concave, then \eqref{eq:sieve-asymp-linear} holds, and $\widehat{\theta}_{m_n, n}$ is best regular in $\mc{P}$.
\end{prop}
\begin{proof} This follows from the Basic Corollary in~\cite{HjortPollard93}. In particular, since $J_{m_n}\to J$ under Assumption~\ref{ass:effscr-approx} we have $A_{m_n,n}(h) = h^{\tr}n^{-1/2}\sum_{i=1}^n \tilde{\ell}_{m_n}(X_i) - \half h^{\tr}J h + o_{P_{m_n}^n}(1)$. As $A_{m,n}(h)$ is concave, $\sqrt{n}(\widehat{\theta}_{m_n,n} - \theta_0) = n^{-1/2} \sum_{i=1}^n J^{-1}\tilde{\ell}_{m_n}(X_i) + o_{P_{m_n}^n}(1)$ by the Basic Corollary.  $\widehat{\theta}_{m_n,n}$ is then best regular in $\mc{P}$ by Theorem~\ref{thm:asymp-linear-efficiency}.
\end{proof}
Compared to Theorem~\ref{thm:asymp-linear-efficiency}, the above proposition allows us to replace~\eqref{eq:sieve-asymp-linear} with~\eqref{eq:Anh_LAN}, provided $h \mapsto A_{m,n}(h)$ is concave. One key advantage of this is that \eqref{eq:Anh_LAN} and concavity give us {\it both} consistency and asymptotic normality of $\widehat{\theta}_{m_n,n}$, while to establish~\eqref{eq:sieve-asymp-linear} without concavity, one typically first needs to establish consistency.

We now present a theorem giving conditions under which~\eqref{eq:Anh_LAN} holds. This theorem is a stripped down and sieved version of Theorem~1 in \citet{murphy2000profile}, and requires that we introduce some quantities inspired by that paper. It is in the construction of these quantities we gain a lot in simplicity by working with growing parametric models, compared to attacking the semiparametric model directly. This is because under $P_m = P_{\theta_0,T_m^{-1}\gamma_0}$, the least favourable submodel for estimating $\theta_0$ always takes the form $\theta \mapsto P_{\theta,T_{m}^{-1}\gamma(\theta)}$, where $\gamma(\theta) = \gamma_0 + i_{m,11}^{-1}i_{m,10}(\theta_0 - \theta)$. In view of this, define $\gamma_t^{\sub}(\theta,\gamma) = \gamma + i_{m,11}^{-1}i_{m,10}(\theta - t)$,
where we suppress the dependence of $\gamma_t^{\sub}(\theta,\gamma)$ upon $m$ from the notation. For each $m$ and each $(\theta,\gamma) \in \Theta \times \Gamma_m$ define the mappings $t \mapsto l_m(t,\theta,\gamma) \coloneqq \log L_m(t,\eta_{\gamma_t^{\sub}(\theta,\gamma)})$.
These functions bridge the gap between the log-profile likelihood and the efficient score. To see this, notice that $\eta_{\gamma_{\theta}^{\sub}}(\theta,\gamma) = \gamma$, and that the derivative of $l_m(t,\theta,\gamma)$ with respect to $t$ is
\begin{equation}
\dot{l}_m(t,\theta,\gamma) = \dot{\ell}_{t,T_m^{-1}\gamma_t^{\sub}(\theta,\gamma)}
- i_{m,01} i_{m,11}^{-1}\dot{v}_{t,T_m^{-1}\gamma_t^{\sub}(\theta,\gamma)}
= \tilde{\ell}_{t,T_m^{-1}\gamma_t^{\sub}(\theta,\gamma) },
\notag
\end{equation}
in particular, $\dot{l}_m(\theta_0,\theta_0,\gamma_0) = \tilde{\ell}_{m}$. At the same time, mimicking the sandwiching technique of \citet{murphy2000profile}, we have, writing $\PP_n f = n^{-1}\sum_{i=1}^n f(X_i)$ for the integral with respect to the empirical measure, that for arbitrary $\tilde{\theta}_n$
\begin{equation}
\PP_n l_{m_n}(\tilde{\theta}_n,\theta_0,\widehat{\gamma}(\theta_0))
- \PP_n l_{m_n} (\theta_0, \theta_0,\widehat{\gamma}(\theta_0))\leq 
\frac{1}{n} (\log \pl_{m_n,n}(\tilde{\theta}_n) - \log \pl_{m_n,n}(\theta_0) ) ,
\label{eq:sandwich_lower}
\end{equation}
and 
\begin{equation}
\frac{1}{n} (\log \pl_{m_n,n}(\tilde{\theta}_n) - \log \pl_{m_n,n}(\theta_0) ) 
\leq \PP_n l_{m_n}(\tilde{\theta}_n,\tilde{\theta}_n,\widehat{\gamma}(\tilde{\theta}_n))
- \PP_n l_{m_n} (\theta_0, \tilde{\theta}_n,\widehat{\gamma}(\tilde{\theta}_n)).
\label{eq:sandwich_upper}
\end{equation}
In particular, the process $h \mapsto A_{m,n}(h)  = \log \pl_{m,n}(\theta_0 + h/\sqrt{n}) - \log \pl_{m,n}(\theta_0)$ can be squeezed between two quantities approximating the {`}efficient LAN expansion{'} of~\eqref{eq:Anh_LAN}.

\begin{theorem}\label{thm:get_Anh_LAN} Set $\tilde{\theta}_{n} = \theta_0 + h/\sqrt{n}$ for some fixed $h$. Assume that $t \mapsto l_m(t,\theta,\gamma)(x)$ is twice continuously differentiable for each $m$, $(\theta,\gamma)$ and $x$, with derivatives $\dot{l}_m$ and $\ddot{l}_m$, and that there is a subsequence $(m_n)$ such that for $\tilde{\psi}$ equal to either $(\tilde\theta_n, \widehat{\gamma}(\tilde{\theta}_n))$ or $(\theta_0, \widehat{\gamma}(\theta_0))$, 
\begin{equation}
\sqrt{n}\PP_n \dot{l}_{m_n}(\theta_0,\tilde{\psi}) = \sqrt{n}\PP_n \tilde{\ell}_{m_n} + o_{P_{m_n}^n}(1), \quad \text{and}\quad 
\PP_n \ddot{l}_{m_n}(s_n, \tilde{\psi}) = - J_{m_n}  + o_{P_{m_n}^n}(1),
\notag
\end{equation}
for any deterministic sequence $s_n \to \theta_0$. Then~\eqref{eq:Anh_LAN} holds.
\end{theorem}
\begin{proof} A Taylor expansion keeping $\tilde{\psi}$ fixed: $n\PP_n l_{m_n}(\theta_0 + h/\sqrt{n}, \tilde{\psi} ) - n\PP_n l_{m_n}(\theta_0, \tilde{\psi} )
= 
h^\tr \sqrt{n}\,\PP_n \dot{l}_{m_n}(\theta_0,\tilde{\psi}) 
+ \half h^\tr\, \PP_n \ddot{l}_{m_n}( s_n, \tilde{\psi}) h$ for an $s_n$ between $\tilde{\theta}_n$ and $\theta_0$. Replace $\PP_n \ddot{l}_{m_n}(s_n, \tilde{\psi})$ by $-J_{m_n} + o_{P_{m_n}^n}(1)$, and $\sqrt{n}\,\PP_n \dot{l}_{m_n}(\theta_0,\tilde{\psi})$ by $\sqrt{n}\PP_n \tilde{\ell}_{m_n} + o_{P_{m_n}^n}(1)$ in both of the sandwiching bounds in~\eqref{eq:sandwich_lower} and~\eqref{eq:sandwich_upper}.
\end{proof}

\subsection{Random quadratic expansions}\label{ssec:sieve-pl:random-quadratic}
The concavity assumption that we worked under in the previous section can be substituted with a consistency assumption. Suppose that for any random sequence $\tilde{\theta}_n$ such that $\tilde{\theta}_n = \theta_0 + o_{P_{m_n}^n}(1)$ for some subsequence $(m_n)$, with $\tilde{h}_n = \tilde{\theta}_n - \theta_0$, we have 
\begin{equation}
    \log \pl_{m_n, n}(\tilde\theta_n) - \log \pl_{m_n,n}(\theta_0) = \tilde{h}_n^\tr \sum_{i=1}^n \tilde{\ell}_{m_n}(X_i) - \half n\tilde{h}_n^\tr J_{m_n} \tilde{h}_n + r_{n}(\tilde\theta_n),
\label{eq:pl-expansion}    
\end{equation}
where $r_{n}(\tilde{\theta}_n) = o_{P_{m_n}^n}(\sqrt{n}\norm{\tilde{h}_n} + n\norm{\tilde{h}_n}^2 + 1)$. This gives a proposition that is similar to Proposition~\ref{theorem:sieve_sandwich}, replacing concavity with consistency. 

\begin{prop}\label{prop:random-pl-expansion-MLE-asymp-linear} Suppose Assumptions \ref{ass:DQM}--\ref{ass:effscr-approx} hold, and let $\widehat{\theta}_{m,n}$ be the maximiser of $\pl_{m,n}(\theta)$. If $\widehat{\theta}_{m_n,n} = \theta_0 + o_{P_{m_n}^n}(1)$ and~\eqref{eq:pl-expansion} holds, then \eqref{eq:sieve-asymp-linear} holds, and consequently $\widehat{\theta}_{m_n,n}$ is best regular in $\mc{P}$. 
\end{prop}
The following theorem, which is analogous to Theorem~\ref{thm:get_Anh_LAN}, provides conditions under which the expansion given in~\eqref{eq:pl-expansion} holds. 

\begin{theorem}\label{thm:sieve-pl} Assume that $t \mapsto l_m(t,\theta,\gamma)(x)$ is twice continuously differentiable for each $m$, $(\theta,\gamma)$ and $x$, with derivatives $\dot{l}_m$ and $\ddot{l}_m$; and that there is a subsequence $(m_n)$ such that for any random sequence $\tilde{\theta}_n$ with $\tilde{\theta}_n = \theta_0 + o_{P_{m_n}^n}(1)$, we have 
\begin{equation}\label{thm:sieve-pl:eq:no-bias}
\sqrt{n}\PP_n \dot{l}_{m_n}(\theta_0,\tilde{\theta}_n,\widehat{\gamma}(\tilde{\theta}_n)) = \sqrt{n}\PP_n \tilde{\ell}_{m_n} + o_{P_{m_n}^n}(\sqrt{n}\norm{\tilde\theta_n - \theta_0} + 1), 
\end{equation}
and 
\begin{equation}\label{thm:sieve-pl:eq:info-eq}
\PP_n \ddot{l}_{m_n}(s_n,\tilde{\theta}_n,\widehat{\gamma}(\tilde{\theta}_n)) = - J_{m_n}  + o_{P_{m_n}^n}(1),
\end{equation}
for any random sequence $s_n = \theta_0 + o_{P_{m_n}^n}(1)$. Then~\eqref{eq:pl-expansion} holds.
\end{theorem}
Proposition \ref{prop:random-pl-expansion-MLE-asymp-linear} and Theorem \ref{thm:sieve-pl} are growing parametric versions of Corollary 1 and Theorem 1 (respectively) in \cite{murphy2000profile}. As they are proven by making what are essentially notational adjustments to the proofs of the cited results, we defer their proofs to Appendix~\ref{sec:pl-conditions}. In Appendix~\ref{sec:pl-conditions} we also provide sufficient conditions for the assumptions made in Theorem~\ref{thm:sieve-pl}.

\section{Applications}\label{sec:applications}
In this section we continue the example of the partially linear model, and also apply our theory to the Cox regression model. Both models satisfy the background assumptions on the semiparametric model (Assumption~\ref{ass:DQM} and \ref{ass:nonsingular-effinfo}), consequently we concentrate on the assumptions made on the parametric approximations, that is Assumptions~\ref{ass:sieve-approx} and \ref{ass:effscr-approx}, along with~\eqref{eq:sieve-asymp-linear}. For both models we take, for simplicity, $\theta \in \Theta \subset \real$, and $\eta$ a real valued $s$-times continuously differentiable function on the unit interval. The parametric approximations employed take the form $\beta_m(z)^{\tr}\gamma = (T_m^{-1}\gamma)(z)$, where for each $m$, $\beta_m = (\beta_{m,1},\ldots,\beta_{m,k_m})^{\tr}$ is a collection of orthonormal functions, and $\gamma = (\gamma_1,\ldots,\gamma_{k_m})^{\tr} \in 
\Gamma_m \subset \real^{k_m}$ are coefficients such that $\beta_m^{\tr}\gamma \to \eta$ in $L_2([0,1],\nu)$, where $\nu$ is an appropriate finite measure (which is always possible, see e.g. Theorem 14.3.1 in \cite{szego1975orthogonal}). To not overburden the notation, the vectors $\gamma$ are not indexed by $m$. Some of the details of the two following examples are left to appendices~\ref{appendix:PLM} and \ref{appendix:cox}. 

We note that the partially linear model considered immediately below is a special case of the class of {`}partially linear GLMs{'} studied by~\cite{mammen1997penalized}, amongst others. It should be uncomplicated to extend our results in Section~\ref{sec:plm} below to a large subclass of such partially linear GLMs. In particular, it is straightforward to establish the concavity of $A_{m, n}$ for any such model with a canonical link function.

\subsection{The partially linear model}\label{sec:plm}
We have $n$ i.i.d.~replicates of $X = (W,Y,Z)$, for an outcome $Y$ and covariates $(W,Z)$ with values in $\real \times [0,1]$. The observation $X$ stems from the model $P_0 = P_{\theta_0,\eta_0}$, where each member $P_{\theta,\eta} \in \mc{P}$ has density
\begin{equation}
p_{\theta,\eta}(x) = \frac{1}{\sqrt{2\pi}\sigma}\exp(-\half (y - \eta(z) - \theta w)^2/\sigma^2 )f_{W,Z}(w,z),
\label{eq:plm_density}
\end{equation}
with respect to Lebesgue measure, and $f_{W,Z}(w,z)$ is the joint density of $(W,Z)$. We assume that $\E\,|W|^{2+\delta}<\infty$ for some $\delta>0$, and denote $P_Z$ the (marginal) law of $Z$. The parametric densities in $\mc{P}_m$ are of the same form as~\eqref{eq:plm_density}, with $\eta$ replaced by $\beta_m^{\tr}\gamma$. The score functions under $\mc{P}_m$ are then 
\begin{equation}
\dot{\ell}_{\theta,T_m^{-1}\gamma}(x) = \frac{w}{\sigma^2}(y - \beta_m(z)^{\tr}\gamma - 
\theta w),\; \text{and}\quad 
\dot{v}_{\theta,T_m^{-1}\gamma}(x) = \frac{\beta_m(z)}{\sigma^2}(y - \beta_m(z)^{\tr}\gamma - 
\theta w).
\label{eq:plm_score}
\end{equation}
Due to the orthonormality of $\beta_m = (\beta_{m,1},\ldots,\beta_{m,k_m})^{\tr}$, the Fisher information matrix under $\mc{P}_m$ takes an appealing form, in particular $i_{m,11} = \sigma^{-2}I_{k_m}$ where $I_{k_m}$ is the $k_m$-dimensional identity matrix, and $i_{m,01} = \sigma^{-2}\, \E\, W \beta_m(Z)^{\tr}$. The efficient score under $\mc{P}_m$ is therefore 
\begin{equation}
\tilde{\ell}_{\theta,T_m^{-1}\gamma}(x) 
= 
\frac{1}{\sigma^2}\big(w - \beta_m(z)^{\tr}\E\,[W \beta_m(Z)] \big)(y - \beta_m(z)^{\tr}\gamma - \theta w).
\notag
\end{equation}
Define $b_0(z) = \E\,(W \given Z = z)$ and  $b_m(z) = \beta_m(z)^{\tr}\E\,[W \beta_m(Z)] =  \sum_{j=1}^{k_m}\beta_{m,j}(z) \langle \beta_{m,j},b_0\rangle$. As the $\beta_m$ form an orthonormal basis for $L_2([0, 1], P_Z)$, we have $b_m\to b_0$ in $L_2([0, 1], P_Z)$. We first turn to Assumption~\ref{ass:sieve-approx}. The log-likelihood ratio of $P_{m}^n$ with respect to $P_0^n$ can be written
\begin{equation}
\log \frac{\dd P_{m_n}^n}{\dd P_0^n} = \frac{1}{\sqrt{n}}\sum_{i=1}^n h_{m,n}(Z_i)(\eps_i/\sigma) - \frac{1}{2n}\sum_{i=1}^n h_{m,n}(Z_i)^2,
\notag
\end{equation}
in terms of $h_{m,n}(z) = (\sqrt{n}/\sigma)(\beta_m(z)^{\tr}\gamma_0 - \eta_0(z))$ and $\eps_i = Y_i - \eta_0(Z_i) - \theta_0 W_i$. Assume that there is a subsequence $(m_n)_{n \geq 1}$ and a function $h$ such that $h_{m_n,n}\to h$ in $L_2(P_0)$. (Conditions under which this holds are given as~\ref{plm-eta-smooth} \&~\ref{plm-eta-rate} in Appendix section~\ref{appendix:PLM}.)
Under this assumption it follows from the fact that the data are i.i.d.~with $\eps$ independent of $Z$ and the law of large numbers that, with $g(x) = h(z)(y - \eta_0(z) - \theta_0 w)/\sigma$, we have 
\begin{equation}
\log \frac{\dd P_{m_n}^n}{\dd P_0^n} = \frac{1}{\sqrt{n}}\sum_{i=1}^n g(X_i) - \half \E_0\, g(X_1)^2 + o_{P_0^n}(1),
\label{eq:plm_RN}
\end{equation}
as $\E_0\, h(Z)^2 = \E_0\, h(Z)^2 (\eps / \sigma)^2$. To conclude that Assumption~\ref{ass:sieve-approx} holds, we need also show that $\E_0 \, g(X) \tilde{\ell}(X) = 0$, which requires that we determine the form of the efficient score. Working formally, we expect that $\tilde{\ell}_{\theta_0,T_m^{-1}\gamma_0}$ will converge to
\begin{equation*}
   u(x)\define \frac{1}{\sigma^2}\big(w - b_0(z) \big)(y - \eta_0(z) - \theta_0 w).
\end{equation*}
We verify that this is true in $L_1(P_0)$ in Appendix~\ref{appendix:PLM}. As $ \tilde{\ell}_{\theta_0,T_{m}^{-1}\gamma_0}(X) \sim \sigma^{-2}(W - b_m(Z))\eps$ under $P_{m}$, $\E\,(|W|^{2+\delta})\,\E\,(|\eps|^{2+\delta})<\infty$ by assumption, and
\begin{equation}
\E\,[\eps^2(b_m(Z) - b_0(Z))^2] = \sigma^2 \E\,[(b_m(Z) - b_0(Z))^2]\to 0,
\label{eq:plm-unif-square-int}
\end{equation}
it follows that $\tilde{\ell}_{\theta_0,T_{m_n}^{-1}\gamma_0}$ is square uniformly $P_{m_n}$-integrable. Since also $P_{m_n} \to P_0$ in total variation by~\eqref{eq:plm_RN} (via~\eqref{eq:sieve-DQM}), Lemma~\ref{lem:L2-conv-varying-measures} in the appendix yields
\begin{equation*}
    \limn \int \|\tilde{\ell}_{\theta_0,T_{m_n}^{-1}\gamma_0} p_{m_n}^{1/2} - up_0^{1/2}\|^2\,\dd\mu =0.
\end{equation*}
Since~\eqref{eq:score-approx-parametric} and~\eqref{eq:score-approx-nonparametric} also hold (as shown in Appendix~\ref{appendix:PLM}), it follows from Corollary~\ref{corollary:to_lemma3.1} that $u = \tilde{\ell}$ $P_0$-almost surely. We now have an expression for the efficient score, and can verify that Assumption~\ref{ass:sieve-approx} holds as
\begin{equation}
\E_0 \, g(X) \tilde{\ell}(X) 
=  \frac{1}{\sigma^2}\E\, h(Z)(W - b_0(Z))  =  \frac{1}{\sigma^2}\E\, h(Z)\E\,\{(W - b_0(Z)) \given Z\} = 0,
\notag
\end{equation}
In Appendix~\ref{appendix:PLM} we show that equations~\eqref{eq:score-approx-parametric} and~\eqref{eq:score-approx-nonparametric} hold, entailing that Assumption~\ref{ass:effscr-approx} also holds by Lemma \ref{lem:score-approx-implies-effscr-approx}. From the developments so far, we see that the semiparametric efficient information is $J = \sigma^{-2}(\E\,W^2 - \norm{b_0}^2) = \sigma^{-2}\E\,\big(W^2 - (\E\,[W \given Z])^2\big)$, and also note that
\begin{equation}
    J_{m_n} = \E_{m_n}\, \tilde{\ell}_{\theta_0,T_{m_n}^{-1}\gamma_0}(X)^2 = \frac{1}{\sigma^2}\big( \E\, W^2 - \sum_{j=1}^{k_{m_n}}|\langle \beta_{m_n,j},b_0\rangle|^2)
\to \frac{1}{\sigma^2}\big( \E\, W^2 - \norm{b_0}^2\big) = J,
\notag
\end{equation}
which may also be verified directly using Parseval{'}s identity.

Let $(\widehat{\theta}_{m,n},\widehat{\gamma}_{m,n})$ be the maximum likelihood estimator under $\mc{P}_m$. To establish that $\widehat{\theta}_{m_n,n}$ is best regular, we show that $h \mapsto \log \pl_{m,n}(\theta_0 + h/\sqrt{n}) - \log \pl_{m,n}(\theta_0)$ is concave and admits a quadratic expansion as in~\eqref{eq:Anh_LAN}, which by Proposition~\ref{theorem:sieve_sandwich} will allow us to conclude that $\widehat{\theta}_{m_n,n}$ is best regular. The value achieving the supremum in $\sup_{\gamma \in \Gamma_m} L_{m,n}(\theta,\gamma)$ is the least squares solution $\widehat{\gamma}(\theta) = B_{m,n}^{-1}n^{-1}\sum_{i=1}^n \beta_m(Z_i)(Y_i - \theta W_i)$, where $B_{m,n} = n^{-1}\sum_{i=1}^n \beta_m(Z_i)\beta_m(Z_i)^{\tr}$. The log-profile likelihood can then be expressed as $\log \pl_{m,n}(\theta) = - 1/(2\sigma^2)\sum_{i=1}^n\big(Y_i - \breve{Y}_{m,n,i} - \theta(W_i - \breve{W}_{m,n,i})\big)^2$, with $\breve{W}_{m,n,i} =  \beta_m(Z_i)^{\tr}B_{m,n}^{-1}n^{-1}\sum_{j=1}^n \beta_m(Z_j) W_j$ and $\breve{Y}_{m,n,i} =  \beta_m(Z_i)^{\tr}B_{m,n}^{-1}n^{-1}\sum_{j=1}^n \beta_m(Z_j) Y_j$. We see that $\log \pl_{m,n}(\theta)$ is concave in $\theta$. The expansion~\eqref{eq:Anh_LAN} is verified Appendix~\ref{appendix:PLM} under conditions~\ref{plm-cond-var-bound} and~\ref{plm-A-LAN-rates}. In summary, under the basic conditions given here along with~\ref{plm-cond-var-bound}--\ref{plm-eta-rate}, the conditions of Proposition~\ref{theorem:sieve_sandwich} are satisfied and consequently the estimator $\widehat{\theta}_{m_n,n}$ is best regular in $\mc{P}$. 

\subsection{Efficiency of the Cox partial likelihood estimator}\label{subsec:cox}
The Cox regression model differs from the partially linear model in an important way. Whilst in the partially linear model a maximum likelihood estimator for $\theta$ cannot be defined without recourse to sieves, penalisation, or the like (as discussed in the Section~\ref{sec::intro}), no such techniques are called for in the Cox regression model, as the Cox partial likelihood permits us to work directly with the semiparametric model. This entails that an estimator for the parametric part of the Cox regression model can be derived straightforwardly, without (for example) the theory developed in this paper. This theory does, however, lead to a simple proof of the efficiency of the Cox partial likelihood estimator, as we will now show. In other words, in this application the theory developed in this paper is used solely as a theoretical tool. Consequently, the choice of basis functions is of no practical importance, and one may (and, indeed, one ought to) use basis functions that make the analysis particularly tractable, which is what we do here.  

Suppose we have $n$ i.i.d.~replicates of $X = (T,\Delta,W)$ observed over $[0,1]$, with $T = \min(T^{\prime},C)$ and $\Delta = I(T^{\prime} \leq C)$, where the lifetime $T^{\prime}$ and the censoring time $C$ are independent given $W$, and $T^{\prime}$ given $W = w$ follows a Cox model, i.e., its hazard rate is of the form $\eta(t) \exp(\theta w)$ where $\theta \in \Theta \subset \real$ for simplicity, and $\eta$ belongs to the space $\mc{H}$ of one time continuously differentiable functions $\eta \colon [0,1] \to (0,\infty)$. As above, $P_{0} = P_{\theta_0,\eta_0}$ is the true data generating mechanism, while $P_m = P_{\theta_0,T_m^{-1}\gamma_0}$ are the parametric approximations. Here we take the basis function 
\begin{equation}
\beta_{m,j}(t) = k_mI_{V_{m,j}}(t),\quad\text{for $j = 1,\ldots,k_m$},
\label{eq:cox_piecewisebasis}
\end{equation}
where $V_{m,1},\ldots,V_{m,k_m}$ are disjoint intervals whose union make up $[0,1]$, and each interval is of length $1/k_m$. With this basis, a natural choice is to work under the coefficient $\gamma_0 = (\gamma_{0,1},\ldots,\gamma_{0,k_m})^{\tr}$ with $\gamma_{0,j} = \eta_0((j-1)/k_m)/k_m$ for $j = 1,\ldots,k_m$. Introduce the counting process $N(t) = \Delta I(T \leq t)$, the at-risk process $Y(t) = I(T \geq t)$, and let $(\mc{F}_{t})_{t \in [0,1]}$ be the filtration generated by these. Then, with respect to this filtration, $M(t) = N(t) - \int_0^t Y(s)\eta_0(s)\exp(\theta_0 W)\,\dd s$ and $M^m(t) = N(t) - \int_0^t Y(s)\beta_m(s)^{\tr} \gamma_0\exp(\theta_0 W)\,\dd s$ are square integrable martingales with respect to $P_0$ and $P_m$, respectively. 

Using the Taylor-expansion $\log(1 + a) = a - \half a^2 + a^2R(a)$ where $R(a) \to 0$ as $a \to 0$, the log-likelihood ratio based on the full sample can be written 
\begin{equation}
\log \frac{\dd P_{m_n}^n}{\dd P_0^n} = \frac{1}{\sqrt{n}}\sum_{i=1}^n \int_0^1 \frac{h_{m_n,n}}{\eta_0}\,\dd M_i
- \frac{1}{2n}\sum_{i=1}^n \int_0^1 \frac{h_{m_n,n}^2}{\eta_0^2}\,\dd N_i + r_{m_n,n}, 
    \notag
\end{equation}
where $h_{m,n} = \sqrt{n}(\beta_m^{\tr}\gamma_0 - \eta_0)$, and $r_{m,n} = n^{-1}\sum_{i=1}^n\int_0^1 h_{m,n}^2/\eta_0^2R(n^{-1/2}h_{m,n}/\eta_0)\,\dd N_i$
Let $(m_n)$ be so that $\sqrt{n}/k_{m_n} \to 0$. Then we can apply Lenglart{'}s inequality (e.g., \citet[p.~86]{andersen1993statistical} or~\citet[Lemma~I.3.30, p.~35]{jacod2003limit}) to show that $\log \dd P_{m_n}^n/\dd P_{0}^n$ tends to zero in probability, and Assumption~\ref{ass:sieve-approx} is satisfied with $g =0$. In fact, with the above choice of basis functions, this is the only possible limit. 

Next, we verify Assumption~\ref{ass:effscr-approx}. Under $\mathcal{P}_m$, the score functions are 
\begin{equation}
\dot{\ell}_m(X)
= W M^m(1),\quad \text{and}\quad
\dot{v}_m(X)
=  \int_0^1 \beta_m(s)(\beta_m(s)^{\tr}\gamma_0)^{-1}\,\dd M^m(s),
    \notag
\end{equation}
when evaluated in $(\theta_0,T_m^{-1}\gamma_0)$. With $s_{m}^{(k)}(t) = \E_m\, Y(t)W^k \exp(W \theta_0)$ for $k = 0,1,2$, this gives 
\begin{equation}
i_{m,01} 
 = \int_0^1 \beta_m(t) s_m^{(1)}(t)\,\dd t, \quad \text{and}\quad 
= \int_0^1 \frac{ \beta_m(t)\beta_m(t)^{\tr}     }{\beta_m(t)^{\tr}\gamma_0}s_m^{(0)}(t)\,\dd t. 
    \notag
\end{equation}
Inserting the locally constant basis functions in~\eqref{eq:cox_piecewisebasis}, we find the efficient score under $\mathcal{P}_m$,
\begin{equation}
\tilde{\ell}_m(X) =  \sum_{j=1}^{k_m}\int_{V_{m,j}}\bigg(W - 
\frac{\int_{V_{m,j}}s_m^{(1)}(u)\,\dd u}{\int_{V_{m,j}}s_m^{(0)}(u)\,\dd u} \bigg)\,\dd M^{m}(t).
    \notag
\end{equation}
 At this point it is tempting to conjecture that with $s^{(k)}$ the pointwise limits of $s_m^{(k)}$, the efficient score under $\mathcal{P}$ is $\tilde{\ell}(X) = \int_0^1 (W - s^{(1)}(t)/s^{(0)}(t) ) \,\dd M(t)$. This is indeed the case, as can be established using Lemma~\ref{lem:score-approx-implies-effscr-approx}, the conditions of which we verify in Appendix~\ref{appendix:cox} by a simple, if tedious, application of Lemma~\ref{lemma:effscr-approx-based-on-Pmn-TV}. 
 
The main message is that these lemmata, combined with Corollary~\ref{corollary:to_lemma3.1}, allow us to conclude that the efficient score under $\mathcal{P}$ is as conjectured above, that is, $\tilde{\ell}(X) = \int_0^1 (W - s^{(1)}/s^{(0)} ) \,\dd M$; and consequently, the efficient information is 
\begin{equation}
J = \E_0\, \tilde{\ell}(X)^2 
= \int_0^1 \bigg(\frac{s^{(2)}(t)}{s^{(1)}(t)} - \frac{s^{(1)}(t)}{s^{(0)}(t)}\bigg)^2 s^{(1)}(t)\eta_0(t)\,\dd t. 
    \notag
\end{equation}
This method of finding the efficient score and the efficient information, leads to a new, relatively simple, and rather parametric proof of the efficiency of the Cox partial likelihood estimator (see for example~\citet[pp.~77--81]{BKRW98} or \citet{kosorok2008introduction} for proofs that differ from ours). Let $L_{n}^{{\rm cox}}(\theta)$ be Cox{'}s partial likelihood function (see \citet{gill1984understanding} for an excellent introduction). We may then define the process 
\begin{equation}
A_{n}^{\rm cox}(h) = \log \frac{L_{n}^{{\rm cox}}(\theta_0 + h/\sqrt{n})}{L_{n}^{{\rm cox}}(\theta_0) } = 
\sum_{i=1}^n \int_0^1 \big( W_i \frac{h}{\sqrt{n}} - \log \frac{S_n^{(0)}(t,\theta_0 + h/\sqrt{n})}{S_n^{(0)}(t,\theta_0)}\big) \,\dd N_i(t),
\notag
\end{equation}
and note that $h \mapsto A_{n}^{\rm cox}(h)$ is concave. Under standard regularity conditions in the Cox regression setting -- given as \ref{cond:cx1}--\ref{cond:cx4} in Appendix~\ref{appendix:cox} -- $A_{n}^{\rm cox}(h)$ admits the expansion $A_{n}^{\rm cox}(h) = hn^{-1/2}\sum_{i=1}^n \tilde{\ell}(X_i) - \half h^2 J + o_{P_0}(1)$, 
where $\tilde{\ell}$ and $J$ are as defined above (this follows from a Taylor expansion combined with the second part of Theorem~3.2 in~\citet[p.~1106]{andersen1982cox}). Since $h \mapsto A_{n}^{\rm cox}(h)$ is concave, the Basic Corollary in~\citet{HjortPollard93} entails that the maximiser of $L_{n}^{{\rm cox}}(\theta)$, say $\widehat{\theta}_n^{\rm cox}$, satisfies~\eqref{eq:semiparam-asymp-linear}, that is $\sqrt{n}(\widehat{\theta}_n^{\rm cox} - \theta_0) = n^{-1/2} \sum_{i=1}^n J^{-1}\tilde{\ell}(X_i) + o_{P_0^n}(1)$.
That the Cox partial likelihood estimator $\widehat{\theta}_n^{\rm cox}$ is best regular under $\mathcal{P}$ now follows from Lemma~25.23 and Lemma~25.25 in \citet[pp.~367--369]{vandervaart1998asymptotic}.

\section{Conclusion}\label{sec:concl}
This paper develops an alternative approach to establishing the asymptotic normality and efficiency of certain approximate maximum likelihood estimators in semiparametric models. These estimators are contiguous sieve estimators, being maximum likelihood estimators in contiguously growing parametric models. The approach detailed in this paper, however, departs substantially from the conventional sieve literature as we work both \emph{with} and \emph{under} the approximating contiguous parametric models, switching back to the semiparametric model in the limit via Le Cam's third lemma. 
Working \emph{with} (growing) parametric models allows for a straightforward definition of a maximum likelihood estimator; working \emph{under} these same parametric models ensures that we may work \emph{as if} our approximating models were correctly specified at each step.
In the examples we have considered in detail, this approach leads to substantial simplifications and relatively straightforward proofs of asymptotic efficiency.

\begin{appendix}
\section{Technical results}\label{app:tech}
\begin{theorem}\label{thm:projection-convergence}
Let $H$ be a Hilbert space, $h_n, h\in H$, and $L_n, L$ closed linear subspaces of $H$. Let $g_n\define \Pi(h_n | L_n)$ and $g\define \Pi(h|L)$. If {\rm(i)} $h_n \to h$ and {\rm(ii)} for each $f\in L$, there is a sequence $(f_n)_{n\in \N}$ and a $N\in \N$ such that $f_n \to f$ and $f_n\in L_n$ for $n\ge N$, then $g_n\to g$.
\end{theorem}
\begin{proof}
Let $\Pi_n$ be the orthogonal projection onto $L_n$ and $\Pi$ that onto $L$. First suppose $h_n=h$ ($n\in \N$). As $(g_n)_{n\in \N}$ is bounded, any subsequence contains a weakly convergent subsequence, say $g_{n_k} \rightharpoonup g^\star$. By self-adjointness and idempotency
\begin{equation}\label{thm:projection-convergence:eq:step1}
        \IP{g_{n_k}}{g_{n_k}} = \IP{\Pi_{n_k} h}{\Pi_{n_k} h} = \IP{h}{\Pi_{n_k} h} \to \IP{h}{g^\star}.
\end{equation}
Let $f\in L$. By hypothesis there are $(f_n)_{n\in \N}$ with $f_n\to f$ and $f_n\in L_n$ for $n\ge N_1$. So $f_{n_k}\to f$ and $f_{n_k}\in L_{n_k}$ for  $k\ge K_1$. Since $h - \Pi_{n_k}h \rightharpoonup h - g^\star$, by Proposition 16.7 in \cite{royden2010real} and the fact that $h - g_{n_k}\in L_{n_k}^\perp$ for each $k$, $\IP{h - g^\star}{f} = \lim_{k\to\infty} \IP{h - g_{n_k}}{f_{n_k}} = 0$. Hence $g^\star = \Pi h = g$. By self-adjointness and idempotency of $\Pi$ and \eqref{thm:projection-convergence:eq:step1}, $ \lim_{k\to\infty}\IP{g_{n_k}}{g_{n_k}} = \IP{h}{\Pi h} = \IP{\Pi h}{\Pi h} = \IP{g}{g}$, and hence $g_{n_k}\to g$ by the Radon--Riesz Theorem. 
As the initial subsequence was arbitrary, $g_n\to g$. To complete the proof, for $h_n\to h$ an arbitrary convergent sequence, $\|g_n - g\| \le \|h_n - h\| + \|\Pi_n h - \Pi h\|$. The first right hand side term is $o(1)$ by assumption; the second by the case with $h_n = h$.    
\end{proof}

\begin{lemma}\label{lem:L2-conv-varying-measures}
    Let $P_n, P_0\ll \mu$ with densities $p_n, p_0$. Suppose that {\rm(i)} $P_n \to P_0$ in total variation; {\rm(ii)} $f_{n}$ converges to $f_0$ in $P_0$-probability; and {\rm(iii)} $f_n$ is uniformly square $P_n$-integrable.
    Then $\limn \int \big(f_n p_{n}^{1/2} - f_0p_0^{1/2}\big)^2\dmu =0$. 
\end{lemma}
\begin{proof}
    $\int f_0^2\darg{P_0} < \infty$ by a version of Fatou's Lemma (e.g. Lemma 2.2 in \cite{serfozo1982}). Expansion of the square yields
    \begin{equation*}
        \int \big(f_n p_{n}^{1/2} - f_0p_0^{1/2}\big)^2\dmu = \int f_n^2 \darg{P_n} + \int f_0^2 \darg{P_0}  -2 \int f_nf_0 p_n^{1/2}p_0^{1/2} \dmu. 
    \end{equation*}
    Combining (i), (ii), (iii) and a version of Vitali's convergence theorem (Corollary 2.9 in \citealp{feinberg2016uniform}) gives $\limn \int f_n^2 \darg{P_n} = \int f_0^2\darg{P_0}$. Hence the proof will be complete if we show that $\limn \int f_nf_0 p_n^{1/2}p_0^{1/2}\dmu = \int f_0^2\darg{P_0}$.

    Let $Q_n$ be the probability measure with $\mu$-density $q_n\define c_n p_{n}^{1/2}p_0^{1/2}$ where $c_n$ is the normalising constant. We have $ c_n^{-1} \define \int  p_{n}^{1/2}p_0^{1/2}\dmu = 1 -\frac{1}{2}\int (p_{n}^{1/2} - p_0^{1/2})^2\dmu \to 1$ as  $\int (p_{n}^{1/2} - p_0^{1/2})^2\dmu\le 2d_{{\rm TV}}(P_n, P_0)$, with $d_{{\rm TV}}$ the total variation distance~\citep[e.g.,][Lemma 2.15]{strasser1985mathematical}. Similarly, 
    \begin{align*}
        \int |q_n - p_0|\dmu 
        &\le \int p_0^{1/2} |c_n| |p_n^{1/2} - p_0^{1/2}|\dmu + \int p_0 |c_n-1|\dmu\\
        &\le |c_n| \bigg(\int p_0 \dmu\bigg)^{1/2} \bigg(\int (p_n^{1/2} - p_0^{1/2})^2\dmu\bigg)^{1/2} + |c_n-1|
        \to 0,
    \end{align*}
    implying that $d_{{\rm TV}}(Q_n, P_0)\to 0$. Now, let $g_n \define f_nf_0$. Note that $h_n \define g_n + |g_n| \ge 0$. Since $|g_n| \to f_0^2$ in $P_0$-probability, we have $\liminf_{n\to\infty} \int |g_n| \darg{Q_n} \ge \int f_0^2 \darg{P_0}$ by a version of Fatou's lemma (Corollary 2.3 in \citealp{feinberg2016uniform}). Additionally, by the Cauchy--Schwarz inequality
    \begin{equation*}
        \limsup_{n\to\infty}\int |g_n|\darg{Q_n} \le \limsup_{n\to\infty} c_n \bigg(\int f_n^2 \darg{P_n} \bigg)^{1/2}\bigg(\int f_0^2 \darg{P_0} \bigg)^{1/2} = \int f_0^2\darg{P_0}.
    \end{equation*}
    In consequence, $\limn \int |g_n|\darg{Q_n} = \int f_0^2\darg{P_0}$.
    An entirely analogous argument applied to $h_n\ge 0$ shows that $\limn \int h_n\darg{Q_n} = 2\int f_0^2\darg{P_0}$.
    Combining these two limit results yields
    \begin{equation*}
        \limn \int f_nf_0 \darg{Q_n} = \limn\int h_n - |g_n| \darg{Q_n} = \int f_0^2\darg{P_0}.
    \end{equation*}
    Since $\int f_nf_n p_n^{1/2}p_0^{1/2}\dmu = c_n^{-1} \int f_nf_0\darg{Q_n}$ and $c_n\to 1$, this completes the proof.
\end{proof}

We next provide proofs of Lemmas~\ref{lemma:effscr-approx-based-on-Pmn-TV} and~\ref{lemma:Fisher_approx_implies_effscr_approx}.
\begin{proof}[Proof of Lemma~\ref{lemma:effscr-approx-based-on-Pmn-TV}]
    Since Assumption \ref{ass:sieve-approx} holds, so does~\eqref{eq:sieve-DQM} and hence $P_{m_n} \to P_0$ in total variation. Apply Lemma \ref{lem:L2-conv-varying-measures}.
\end{proof}

\begin{proof}[Proof of Lemma~\ref{lemma:Fisher_approx_implies_effscr_approx}]
   In case (i) we have $f_{m_n}p_{m_n}^{1/2} \to f_0p_0^{1/2}$ in $\mu$-measure. The conclusion follows from Proposition 2.29 in \cite{vandervaart1998asymptotic} as $\limsup_{n}\int (f_{m_n}p_{m_n}^{1/2})^2\dmu \le \int (f_0p_0^{1/2})^2\dmu <\infty$. In case (ii) the result follows directly from Proposition S3.1 in the supplementary material to \cite{hoesch24locally}.
\end{proof}

\section{Additional details for Section~\ref{ssec:sieve-pl:random-quadratic}}\label{sec:pl-conditions} 

\subsection{Proofs of Proposition \ref{prop:random-pl-expansion-MLE-asymp-linear} and Theorem \ref{thm:sieve-pl}}

\begin{proof}[Proof of Proposition \ref{prop:random-pl-expansion-MLE-asymp-linear}] Let $\Delta_{n}\define n^{-1/2}\sum_{i=1}^n \tilde{\ell}_{m_n}(X_i)$ and $\widehat{h}_{n}\define \sqrt{n}(\hat{\theta}_{n, m_n} - \theta)$. Applying~\eqref{eq:pl-expansion} with $\tilde\theta_n = \widehat{\theta}_{m_n,n}$ gives
    \begin{equation*}
        \log \pl_{n, m_n}(\widehat{\theta}_{m_n,n}) = \log \pl_{n, m_n}(\theta_0) + \widehat{h}_n^\tr \Delta_n - \frac{1}{2}\widehat{h}_n^\tr J_{m_n}\widehat{h}_n + r_{n}(\hat{\theta}_{n, m_n}),
    \end{equation*}
    where $r_{n}(\widehat{\theta}_{m_n,n}) = o_{P_{m_n}^n}(\|\widehat{h}_n\| + 1)^2$. Since $J_{m_n} = \E_{m_n}\tilde{\ell}_{m_n}(X)\tilde{\ell}_{m_n}(X)^\tr$, Assumptions \ref{ass:nonsingular-effinfo}~\&~\ref{ass:effscr-approx} ensure that $J_{m_n}^{-1}\Delta_n =  O_{P_{m_n}^n}(1)$. By~\eqref{eq:pl-expansion} with $\tilde\theta_n = \theta + n^{-1/2}J_{m_n}^{-1}\Delta_n$ and  $r_{n}(\tilde\theta_n) = o_{P_{m_n}^n}(1)$, we have $\log \pl_{n, m_n}(\tilde\theta_n)= \log \pl_{n, m_n}(\theta_0) + \Delta_n^\tr J_{m_n}^{-1}\Delta_n - \frac{1}{2}\Delta_n^\tr J_{m_n}^{-1}\Delta_n + r_{n}(\tilde{\theta}_{n})$. By definition, $\log \pl_{n, m_n}(\widehat{\theta}_{m_n,n})$ is larger than $\log \pl_{n, m_n}(\tilde\theta_{n})$. Hence 
    \begin{equation*}
        \widehat{h}_n^\tr \Delta_n - \frac{1}{2}\widehat{h}_n^\tr J_{m_n}\widehat{h}_n  - \frac{1}{2}\Delta_n^\tr J_{m_n}^{-1}\Delta_n \ge -o_{P_{m_n}^n}(\|\widehat{h}_n\| + 1)^2.
    \end{equation*}
    The left hand side of the preceding display is equal to the left hand side of
    \begin{equation*}
        -\frac{1}{2}\left(\widehat{h}_n - J_{m_n}^{-1}\Delta_n\right)^\tr J_{m_n}\left(\widehat{h}_n - J_{m_n}^{-1}\Delta_n\right) \le -\frac{c}{2} \|\widehat{h}_n - J_{m_n}^{-1}\Delta_n\|^2,
    \end{equation*}
    where $0 < c \le \lambda_{\min}(J_{m_n})$ for all sufficiently large $n$, where $\lambda_{\min}$ is the smallest eigenvalue of $J_{m_n}$, and we use that these are bounded below for $n$ sufficiently large. Combination of the preceding two displays yields
    \begin{equation*}
        \|\widehat{h}_n -J_{m_n}^{-1}\Delta_n\| = 
        o_{P_{m_n}^n}(\|\widehat{h}_n\| + 1).
    \end{equation*}
Since $\|J_{m_n}^{-1}\Delta_n\| = O_{P_{m_n}^n}(1)$, the triangle inequality implies that 
    \begin{equation*}
        \|\widehat{h}_n\| \le \|\widehat{h}_n - J_{m_n}^{-1}\Delta_n\| + \|J_{m_n}^{-1}\Delta_n\| = o_{P_{m_n}^n}(\|\widehat{h}\|_n + 1) + O_{P_{m_n}^n}(1) =O_{P_{m_n}^n}(1).
    \end{equation*}
Using this in the penultimate display yields $\|\widehat{h}_n - J_{m_n}^{-1}\Delta_n\| = o_{P_{m_n}^n}(1)$, implying \eqref{eq:sieve-asymp-linear}.
\end{proof}

\begin{remark} In the proof of Proposition~\ref{prop:random-pl-expansion-MLE-asymp-linear} the expansion~\eqref{eq:pl-expansion} is used for two different $\tilde\theta_n$, namely $\tilde\theta_n = \widehat{\theta}_{m_n,n}$ and $\tilde\theta_n \define \breve{\theta}_{m_n, n} \define n^{-1/2}J_{m_n}^{-1}\Delta_n$, where $\Delta_n = n^{-1/2}\sum_{i=1}^n\effscrarg{m_n}(X_i)$. 
    Under Assumptions \ref{ass:DQM}-\ref{ass:effscr-approx}, $\sqrt{n}(\breve{\theta}_{m_n, n} - \theta_0) = J_{m_n}^{-1}\Delta_n = O_{P_{m_n}^n}(1)$,
    as noted in the proof. Therefore, if one establishes that also $\sqrt{n}(\widehat\theta_{m_n,n} - \theta) =O_{P_{m_n}}(1)$, then it suffices to show that $r_{n}(\tilde\theta_n) = o_{P_{m_n}}(1)$ for all $\tilde{\theta}_n$ such that $\sqrt{n}(\tilde\theta_n - \theta) = O_{P_{m_n}}(1)$.
\end{remark}

\begin{proof}[Proof of Theorem \ref{thm:sieve-pl}] Let $\tilde{h}_n = \tilde{\theta}_n - \theta_0$. For fixed $\tilde{\psi}$, a Taylor expansion yields $\PP_n l_{m_n}(\tilde{\theta}_n,\tilde{\psi})
- \PP_n l_{m_n}(\theta_0,\tilde{\psi})
 = \tilde{h}_n^{\tr}\PP_n \dot{l}_{m_n}(\theta_0 ,\tilde{\psi}) + \half \tilde{h}_n^{\tr} \PP_n \ddot{l}_{m_n}(s_n,\tilde{\psi})\tilde{h}_n$,
for a $s_n$ between $\tilde\theta_n$ and $\theta_0$. Multiplying through by $n$ and replacing $\PP_n \ddot{l}_{m_n}(s_n, \tilde{\psi})$ by $-J_{m_n} + o_{P_{m_n}^n}(1)$ and $\sqrt{n}\,\PP_n \dot{l}_{m_n}(\theta_0,\tilde{\psi})$ by $\sqrt{n}\PP_n \tilde{\ell}_{m_n} + o_{P_{m_n}^n}(\sqrt{n}\|\tilde{h}_n\|+1)$ on the right hand side gives
\begin{equation}
    n\PP_n l_{m_n}(\tilde{\theta}_n,\tilde{\psi})
- n\PP_n l_{m_n}(\theta_0,\tilde{\psi})
 = \tilde{h}_n^{\tr}\sumin \effscrarg{m_n}(X_i) - \half n\tilde{h}_n^{\tr} J_{m_n}\tilde{h}_n + o_{P_{m_n}^n}((\sqrt{n}\|\tilde{h}_n\|  + 1)^2),
 \notag
\end{equation}
for $\tilde{\psi}$ equal to either $(\tilde\theta_n, \widehat{\gamma}(\tilde\theta_n)$ or $(\theta_0, \widehat{\gamma}(\theta_0))$. Applying the sandwiching bounds in~\eqref{eq:sandwich_lower} and~\eqref{eq:sandwich_upper} gives~\eqref{eq:pl-expansion}.
\end{proof}

\subsection{Sufficient conditions for Theorem~\ref{thm:sieve-pl}} As demonstrated in the examples in the main text, in some models, our approach of working under the parametric models $P_{m}$ allows \eqref{eq:sieve-asymp-linear} to be established directly. For cases where this is not possible, Proposition~\ref{prop:random-pl-expansion-MLE-asymp-linear} and Theorem~\ref{thm:sieve-pl} provide a general result for such estimators. This result is based on \cite{murphy2000profile} with the key difference being that we consider a sieved profile likelihood in which, at each step $m$, true least-favourable submodels necessarily exist.  This avoids the requirement to construct {`}approximate least favourable submodels{'} as in \cite{murphy2000profile}. Nevertheless, as the theoretical analysis of this contiguous sieve profile likelihood estimator proceeds very similarly to the analysis of the semiparametric profile likelihood estimator considered in \cite{murphy2000profile}, Theorem \ref{thm:sieve-pl} states the result under high level conditions. Here we give lower-level structural conditions which imply the conditions~\eqref{thm:sieve-pl:eq:no-bias} and~\eqref{thm:sieve-pl:eq:info-eq} required by Theorem \ref{thm:sieve-pl}. The conditions are  similar to those given by \cite{murphy2000profile}, but require some adjustment as we deal with a sequence of parametric likelihoods. 

The following lemma splits condition~\eqref{thm:sieve-pl:eq:no-bias} into a no-bias condition and a condition relating to the empirical process $\G$, here defined by $\G f = \sqrt{n}(\EP f - P_{m_n}f)$.  

\begin{lemma}\label{lem:no-bias-Donsker}
    Suppose that for any $\tilde\theta_n = \theta_0 + o_{P_{m_n}^n}(1)$, 
     \begin{equation}\label{lem:no-bias-Donsker:eq:no-bias}
            \E_{m_n}\dot{l}_{m_n}(\theta_0, \tilde\theta_n, \widehat{\gamma}(\tilde\theta_n)) =o_{P_{m_n}^n}(\|\tilde\theta_n - \theta_0\| + n^{-1/2})~,
        \end{equation}
        and 
        \begin{equation}\label{lem:no-bias-Donsker:eq:Donsker}
            \G \dot{l}_{m_n}(\theta_0, \tilde\theta_n, \widehat{\gamma}(\tilde\theta_n)) = \G\effscrarg{m_n} + o_{P_{m_n}^n}(1).  
        \end{equation}
Then~\eqref{thm:sieve-pl:eq:no-bias} holds.   
\end{lemma}
\begin{proof}\belowdisplayskip=-12pt By~\eqref{lem:no-bias-Donsker:eq:no-bias} and \eqref{lem:no-bias-Donsker:eq:Donsker}, and since $\E_{m_n}\effscrarg{m_n} = 0$, 
    \begin{align*}
        \sqrt{n}\EP\dot{l}_{m_n}(\theta_0, \tilde\theta_n, \widehat{\gamma}(\tilde\theta_n)) &= \G\dot{l}_{m_n}(\theta_0, \tilde\theta_n, \widehat{\gamma}(\tilde\theta_n)) + \sqrt{n}\E_{m_n} \dot{l}_{m_n}(\theta_0, \tilde\theta_n, \widehat{\gamma}(\tilde\theta_n))\\
        &= \G\effscrarg{m_n} + o_{P_{m_n}^n}(\sqrt{n}\|\tilde\theta_n - \theta_0\| +1)\\
        &=\sqrt{n}\EP \effscrarg{m_n}+ o_{P_{m_n}^n}(\sqrt{n}\|\tilde\theta_n - \theta_0\| +1).
    \end{align*}
\end{proof}
Condition~\eqref{lem:no-bias-Donsker:eq:no-bias} is a {`}no-bias{'} condition, cf.~the discussion in \cite{murphy2000profile}. Condition~\eqref{lem:no-bias-Donsker:eq:Donsker} can be shown to hold if the nuisance parameter estimator $\widehat{\gamma}$ satisfies a consistency condition and a stochastic equicontinuity type condition is satisfied. 

For any $m$, $\gamma = \gamma_m\in \R^{k_m}$ may be viewed as an eventually zero sequence in $\R^\N$. Let $\Gamma_{m}^\N$ denote the subset of $\R^\N$ corresponding to vectors in $\Gamma_m$. We equip $\R^\N$ with a topology, $\tau$, and shall need the following consistency condition to hold: For any $\tilde\theta_n = \theta_0 + o_{P_{m_n}^n}(1)$, 
\begin{equation}\label{eq:eta-profile-consistency}
  \widehat{\gamma}(\tilde\theta_n) -\gamma_0 = o_{P_{m_n}^n}(1). 
\end{equation}
That is, for any neighbourhood $\mc{U}$ of zero in $(\R^\N, \tau)$, $\limn P_{m_n}( \widehat{\gamma}(\tilde\theta_n) -\gamma_0 \in \mc{U})=1$, 
for $\widehat{\gamma}(\tilde\theta_n), \gamma_0$ considered as elements in $\R^\N$.
The topology $\tau$ is arbitrary. It may, for instance, be that induced by $\mc{H}$ (if $\mc{H}$ is a topological vector space). However, for this condition to be useful, the topology needs to be strong enough to imply certain continuity conditions.

\begin{lemma} Suppose that $\tilde\theta_n = \theta_0 + o_{P_{m_n}^n}(1)$ and that~\eqref{eq:eta-profile-consistency} holds. 
    Let $\mc{V}$ be a neighbourhood of $(0, 0)\in \R^{p}\times \R^\N$.
    Suppose that that for each $\eps, \upsilon >0$ there are random $\Delta_n(\eps, \upsilon)\ge0$ and $N(\varepsilon, \upsilon)$ such that if $n\ge N(\varepsilon, \upsilon)$ then {\rm(i)} $P_{m_n}(\Delta_n(\epsilon, \upsilon) > \upsilon)<\varepsilon$  and {\rm(ii)} 
    \begin{equation*}
        \sup_{v\in \mc{V}_n}\| \G \dot{l}_{m_n}(\theta_0, \psi_0 + v ) - \G \dot{l}_{m_n}(\theta_0, \psi_0) \|\le \Delta_n(\varepsilon, \upsilon), \qquad \psi_0 \define (\theta_0, \gamma_0),
    \end{equation*}
    where 
    $\mc{V}_n\define \{v\in \mc{V} : \psi_0 + v \in \Theta\times \Gamma_{m_n}^\N\}$.
    Then~\eqref{lem:no-bias-Donsker:eq:Donsker} holds.
\end{lemma}
\begin{proof}
    Let $X_n(v) \define \G \dot{l}_{m_n}(\theta_0, \psi_0 + v)$ and note that $X_n(0) = \G \dot{l}_{m_n}(\theta_0, \psi_0  ) = \G \effscrarg{m_n}$. Fix $\varepsilon, \upsilon>0$. By~\eqref{eq:eta-profile-consistency} there is a $N_1$ such that $n\ge N_1$ implies $P_{m_n}\left(\hat{v}_n\in \mc{V}\right) \ge 1 -\varepsilon/2$, where $\hat{v}_n \define (\tilde\theta_n - \theta_0, \widehat{\gamma}(\tilde\theta_n) - \gamma_0)$. Note that, by definition, if $\hat{v}_n\in \mc{V}$ then $\hat{v}_n\in \mc{V}_n$. 
    Using this, (i) and (ii) we have for all $n\ge \max\{N_1, N(\varepsilon/2, \upsilon)\}$,
    \begin{align*}
        P_{m_n}\left(\|X_n(\hat{v}) - X_n(0)\| >\upsilon\right) &\le  P_{m_n}\left(\hat{v}_n\notin \mc{V}\right) + P_{m_n}\left(
        \|X_n(\hat{v}) - X_n(0)\| >\upsilon, \hat{v}_n \in \mc{V}_n
        \right)\\
        &< \varepsilon/2 + P_{m_n}\big(\sup_{v\in \mc{V}_n} \|X_n(v) - X_n(0)\|>\upsilon\big)\\
        &\le \varepsilon/2 + P_{m_n}\left(\Delta_n(\varepsilon/2, \upsilon) > \upsilon\right)<\varepsilon,
    \end{align*}
    as required.
\end{proof}
Finally, condition~\eqref{thm:sieve-pl:eq:info-eq} is an approximate information equality. Provided the information equality approximately holds in the least-favourable parametric submodels, this will hold under continuity and moment conditions. 

\begin{lemma}\label{lem:sufficient-cond-information}
    Suppose that $\tilde\vartheta_n = \theta_0 + o_{P_{m_n}^n}(1)$, $\tilde\theta_n = \theta_0 + o_{P_{m_n}^n}(1)$, that~\eqref{eq:eta-profile-consistency} holds and
    \begin{enumerate}[label = {\rm(\roman*)}]
    \item \label{lem:sufficient-cond-information:itm:info-eq}$P_m [\dot{l}_{m}(\theta_0, \theta_0, \gamma_0)\dot{l}_{m}(\theta_0, \theta_0,\gamma_0)^\tr] = -P_{m}[\ddot{l}_{m}(\theta_0, \theta_0, \gamma_0)] + o(1)$ as $m\to\infty$;
        \item \label{lem:sufficient-cond-information:itm:continuity} $(\tilde\vartheta_n, \tilde\theta_n,  \widehat{\gamma}(\tilde\theta_n)) - (\theta_0, \theta_0, \gamma_0)= o_{P_{m_n}^n}(1)$ implies that \begin{equation*}
    \limn \E_{m_n}\|\ddot{l}_{m_n}(\tilde\vartheta_n, \tilde\theta_n, \widehat{\gamma}(\tilde\theta_n))- \ddot{l}_{m_n}(\theta_0, \theta_0, \gamma_0)\|= 0; 
    \end{equation*}
    \item $\|\ddot{l}_{m_n}(\theta_0, \theta_0, \gamma_0)\|$ is uniformly $P_{m_n}$-integrable.
    \label{lem:sufficient-cond-information:itm:UI}
    \end{enumerate}
    Then \eqref{thm:sieve-pl:eq:info-eq} holds.
 \end{lemma}
\begin{proof} By the construction of the least favourable submodels and~\ref{lem:sufficient-cond-information:itm:info-eq},
\begin{equation*}
    J_{m} = \E_m\,\effscrarg{m}\effscrarg{m}^\tr = \E_m\,\dot{l}_{m}(\theta_0, \theta_0, \gamma_0)\dot{l}_{m}(\theta_0, \theta_0,\gamma_0)^\tr = - \E_m\,\ddot{l}_{m}(\theta_0, \theta_0, \gamma_0) + o(1).
\end{equation*}
Write $Y_{n, i}\define \ddot{l}_{m_n}(\theta_0, \theta_0, \gamma_0)(X_i)$ and $\tilde{Y}_{n, i}\define \ddot{l}_{m_n}(\tilde\vartheta_n, \tilde\theta_n, \widehat{\gamma}(\tilde\theta_n))(X_i)$. Then~\eqref{thm:sieve-pl:eq:info-eq} holds if both (a) $n^{-1}\sum_{i=1}^n (\tilde{Y}_{n,i} - Y_{n, i})$ and (b) $n^{-1}\sum_{i=1}^n( Y_{n,i} - \E_{m_n}Y_{n, i})$ are $o_{P_{m_n}^n}(1)$. Requirement
(a) follows as condition~\ref{lem:sufficient-cond-information:itm:continuity} implies $\E_{m_n}\|\tilde{Y}_{n,i} - Y_{n, i}\| \to 0$ and hence $n^{-1}\sum_{i=1}^n( \tilde{Y}_{n,i} - Y_{n, i})$ converges to zero in probability by Markov's inequality. (b) holds as  
under condition \ref{lem:sufficient-cond-information:itm:UI} $n^{-1}\sum_{i=1}^n( Y_{n,i} - \E_{m_n}Y_{n, i})$ converges to zero by the weak law of large numbers \cite[see, e.g.,][]{gut1992weaklaw}.
\end{proof}

\section{The applications}\label{appendix:applications} 

\subsection{The partially linear model}\label{appendix:PLM}
In this section we first verify that~\eqref{eq:score-approx-parametric} and~\eqref{eq:score-approx-nonparametric} hold in the setup of Section~\ref{sec:plm}. This entails, by way of Lemma~\ref{lem:score-approx-implies-effscr-approx}, that the partially linear model satisfies Assumption~\ref{ass:effscr-approx}. We then establish that $\tilde{\ell}_{\theta_0,T_m^{-1}\gamma_0}$ converges to the semiparametric efficient score. 

Consider submodels of the form $\tau \mapsto p_{\theta_0 + a \tau, \eta_0 + b \tau}$ where $a \in \real$ and $b \in B \subset \mc{H}$. Differentiating with respect to $\tau$, and evaluating in $\tau = 0$
\begin{equation}
\frac{\dd}{\dd \tau} \log p_{\theta + a \tau, \eta + b \tau} \big|_{\tau = 0} = a \dot{\ell}_{\theta,\eta} + \frac{b(z)}{\sigma^2}( y - \eta(z) - \theta w), 
\notag
\end{equation}
where $\dot{\ell}_{\theta,\eta}(x) = \sigma^{-2}w(y - \eta(z) - \theta w)$.
We now show that the two terms on the right are the limits, in the sense of~\eqref{eq:score-approx-parametric} and~\eqref{eq:score-approx-nonparametric}, of their parametric counterparts $\dot{\ell}_{\theta,T_m^{-1}\gamma}$ and $\dot{v}_{\theta,T_m^{-1}\gamma}$ given in~\eqref{eq:plm_score}. To this end, we use Lemma~\ref{lem:L2-conv-varying-measures}, and note for future reference that $d_{\rm TV}(P_{m_n}, P_0)\to 0$ by Assumption~\ref{ass:sieve-approx} and \eqref{eq:sieve-DQM}, and that for any $b \in B$, the sequence defined by $\tilde{b}_m(z) = \beta_m^\tr(z)\E\,\{b(Z) \beta_{m}(Z)\}$ is such that $\tilde{b}_m \to b$ in $L_2(P_Z)$. Note also that under $P_m$, $\dot{\ell}_{m}(X) \sim \sigma^{-2}W\eps$ and with $\tilde{\gamma} = \E\,\{b(Z)\beta_m(Z)\}$, $\tilde{\gamma}^\tr \dot{v}_{m}(X) \sim  \sigma^{-2}\tilde{b}_m(Z)\eps$. The uniform square $P_{m_n}$-integrability required by Lemma~\ref{lem:L2-conv-varying-measures} then follows from the integrability condition on $W$ and~\eqref{eq:plm-unif-square-int}. Additionally, it is straightforward to check that $\dot{\ell}_{\theta, \eta}(X)$ and $b(Z)(Y - \eta(Z) - \theta W)$ are square integrable under $P_0$. Finally we show $L_1(P_0)$ convergence of the parametric scores to the semiparametric scores. In particular, 
\begin{equation*}
    \int \|\dot{\ell}_m -\dot{\ell} \|\darg{P_0}
    =\sigma^{-2} \,\E\, \|W (\eta_0(Z) - \eta_m(Z))\| \to 0,
\end{equation*}
by the Cauchy--Schwarz inequality as $\eta_m(z) \define \beta_m(z)^\tr\gamma_0 \to \eta_0(z)$ in $L_2(P_Z)$. Similarly with $\tilde{\gamma}= \E\,\{b(Z)\beta_m(Z)\}$ as above, 
\begin{equation*}
    \int \|\tilde{\gamma}^\tr\dot{v}_{m} -Db \|\darg{P_0}= \sigma^{-2} \E \,\|\epsilon(\tilde{b}_m(Z) - b(Z)) +  \tilde{b}_m(Z)(\eta_0(Z) - \eta_m(Z))\| \to 0,
\end{equation*}
as $\tilde{b}_m(z)\to b(z)$ in $L_2(P_Z)$. Applying Lemma \ref{lem:L2-conv-varying-measures} then verifies~\eqref{eq:score-approx-parametric} and~\eqref{eq:score-approx-nonparametric} of Lemma~\ref{lem:score-approx-implies-effscr-approx}, meaning that Assumption~\ref{ass:effscr-approx} holds. 

To verify that $\effscrarg{m}$ converges to $\tilde{\ell}$ in $L_1(P_0)$ as claimed in Section~\ref{sec:plm} we note that similarly,
\begin{equation*}
    \int \|\effscrarg{m} - u\|\darg{P_0} = \sigma^{-2} \E\,\{\| \eps (b_0(Z) - b_m(Z)) + (W-b_m(Z))(\eta_0(Z) - \eta_m(Z))\} \to 0.
\end{equation*}
To establish~\eqref{eq:Anh_LAN}, we use Theorem~\ref{thm:get_Anh_LAN}. We have 
\begin{equation*}
    l_m(t, \theta, \gamma)(x)=
     - \frac{1}{2\sigma^2}\left(y - w t - \beta_m(z)^\tr ( \gamma + \E\,\{\beta_m(Z)W\}(\theta - t))\right)^2 + C(w, z),
\end{equation*}
for $C(w, z)$ a term which does not depend on $(t, \theta, \gamma)$. Thus
\begin{align*}
    & \dot{l}_m(t, \theta, \gamma)(x) = \frac{1}{\sigma^2}\left(y - w t - \beta_m(z)^\tr ( \gamma + \E\,\{\beta_m(Z)W\}(\theta - t))\right)\\
    & \qquad \qquad\qquad \qquad \times \left(
    w - \beta_m(z)^\tr \E\,\{\beta_m(Z)W\}\right),
\end{align*}
and
\begin{equation*}
    \ddot{l}_{m}(t, \theta, \gamma)(x) = -\frac{1}{\sigma^2}\left(w - \beta_m(z)^\tr \E\,\{\beta_m(Z)^\tr W\}\right)^2 = -\sigma^{-2}\left(w - b_m(z)\right)^2.
\end{equation*}
We first note that $\E_m \ddot{l}_m(t, \theta_0, \gamma_0) = -J_{m}$. Thus for the second part of the condition in Theorem~\ref{thm:get_Anh_LAN} it suffices to show that $n^{-1}\sum_{i=1}^n \sigma^{-2}(W_i - b_{m_n}(Z_i))^2 - J_{m_n} =o_{P_{m_n}^n}(1)$. This follows from the weak law of large numbers as the $\sigma^{-2}(W_i - b_{m_n}(Z_i))^2$ are uniformly $P_{m_n}$-integrable (see \citet{gut1992weaklaw}). To establish the first part of Theorem~\ref{thm:get_Anh_LAN}, let $h\in \R$ and set $\tilde\theta_n = \theta_0 + h / \sqrt{n}$. We then have
\begin{align*}
& \dot{l}_{m_n}(\theta_0, \tilde\theta_n, \hat{\gamma}(\tilde\theta_n))(X_i)\\
     & \qquad = \frac{1}{\sigma^2}
    \bigg(Y_i - W_i \theta_0 - \beta_{m_n}(Z_i)^\tr \bigg[\widehat{\gamma}(\tilde\theta_n)+ \E\,\{\beta_m(Z)W\}\frac{h}{\sqrt{n}}\bigg]\bigg)\left(W_i - b_m(Z_i)\right),
\end{align*}
where $\widehat{\gamma}(\theta)$ and $B_{m,n}$ are as defined in Section~\ref{sec:plm}. 
Therefore, the difference 
\begin{align*}
& \dot{l}_{m_n}(\theta_0, \tilde\theta_n, \widehat{\gamma}(\tilde\theta_n))(X_i) - \tilde{\ell}_{m_n}(X_i) \\
      &\; = \frac{1}{\sigma^2}[W_i - b_{m_n}(Z_i)]\beta_{m_n}(Z_i)^\tr\\ 
      & \qquad \times \left[
     (\widehat{\gamma}(\theta_0) - \gamma) + \frac{h}{\sqrt{n}}\left(\E\,\{\beta_{m_n}(Z)W\} - B_{m_n, n}^{-1}\meanin \beta_{m_n}(Z_i) W_i\right)
     \right].
\end{align*}
In consequence, to verify the remaining condition of Theorem~\ref{thm:get_Anh_LAN}, it suffices to show that
\begin{equation}
\begin{split}
    &\fracrootn\sumin (W_i - b_{0}(Z_i))\beta_{m_n}(Z_i)^\tr (\widehat{\gamma}(\theta_0) - \gamma); \\
    &\fracrootn\sumin  (b_{0}(Z_i)- b_{m_n}(Z_i))\beta_{m_n}(Z_i)^\tr (\widehat{\gamma}(\theta_0) - \gamma);\\
    &\meanin (W_i - b_{m_n}(Z_i))\beta_{m_n}(Z_i)^\tr \left[\E\,\{\beta_{m_n}(Z)W\} - B_{m_n, n}^{-1}\meanin \beta_{m_n}(Z_i)W_i\right],
\end{split}
\label{eq:three_conditions}
\end{equation}
are all $o_{P_{m_n}}(1)$. In order to verify this  we impose the following assumptions (which may be relaxed if more conditions are imposed upon $W$, see below). Let $\xi_{m}\define \sup_{z\in [0, 1]}\|\beta_{m_n}(z)\|$ and let $a_n$ be such that $\|b_0(Z) - b_{m_n}(Z)\|_{L_2(\nu)} \le a_{n}$. (Upper bounds on $a_n$ are available from approximation theory under, for example, smoothness conditions on $b_0$.) 
\begin{enumerate}[label=(pl\arabic*), series=plmconditions]\itemsep-0.2em
    \item \label{plm-cond-var-bound}$\E\,(W^2\given Z) \le C$, $P_Z$-almost surely;
    \item \label{plm-A-LAN-rates}
    $k_{m_n}^2 / \sqrt{n}\to 0$, $\xi_{m_n}k_{m_n} / \sqrt{n}\to 0$ and $a_{n}\sqrt{k_{m_n}}\xi_{m_n} \to 0$.
\end{enumerate}
For the first sum in \eqref{eq:three_conditions}, note that as for $k\neq i$
\begin{equation*}
    \E\,\left\{
        (W_i - b_0(Z_i))\beta_{m_n, j}(Z_i) (W_k - b_0(Z_k))\beta_{m_n, j}(Z_k)
    \right\} = 0
\end{equation*}
then for some positive constant $C_0$
\begin{align*}
        P_{m_n}\big(
        \sumin (W_i - b_{0}(Z_i))\beta_{m_n, j}(Z_i) \ge \sqrt{n}M    \big) 
        &\le \frac{2\E\,(\E\,(W^2 \given Z) + b_0(Z)^2) \beta_{m_n, j}(Z)^2}{M}\le \frac{C_0}{M},
    \end{align*}
which can be made arbitrary small by taking $M$ large enough. We also have
\begin{equation}\label{eq:plm-eq-belloni}
    \sum_{j=1}^{k_{m_n}} |\widehat{\gamma}_j(\theta_0) - \gamma_j| \le \sqrt{k_{m_n}}
    \big(\sum_{j=1}^{k_{m_n}}|\widehat{\gamma}_j(\theta_0) - \gamma_j|^2\big)^{1/2}
    =O_{P_{m_n}}(k_{m_n} / \sqrt{n}) 
\end{equation}
by the Cauchy--Schwarz inequality and Theorem~4.1 in~\citet{belloni2015some}. It follows that by taking $M =  M_n =k_{m_n}$ above and using the union bound that 
\begin{equation*}
     \sum_{j=1}^{k_{m_n}} (\hat{\gamma}_j(\theta_0) - \gamma_j) \fracrootn\sum_{i=1}^{n} \beta_{m_n, j}(Z_i) (W_i - b_0(Z_i)) = 
     O_{P_{m_n}}(k_{m_n}^2 n^{-1/2})= 
     o_{P_{m_n}}(1),
\end{equation*}
which verifies the first term. The second holds as by Theorem 4.1 in \cite{belloni2015some} again, 
\begin{align*}
     &\big|\fracrootn\sumin  (b_{0}(Z_i)- b_{m_n}(Z_i))\beta_{m_n}(Z_i)^\tr (\hat{\gamma}(\theta_0) - \gamma)\big|\\ 
     & \qquad \le \sqrt{n}\xi_{m_n} \|\hat{\gamma}(\theta_0) - \gamma\| \meanin |b_{0}(Z_i)- b_{m_n}(Z_i)|\\
     &\qquad = O_{P_{m_n}}\left(\sqrt{n}\xi_{m_n}\sqrt{k_{m_n}}n^{-1/2}a_n 
     \right) = O_{P_{m_n}}\left(\xi_{m_n}\sqrt{k_{m_n}}a_n \right)  = o_{P_{m_n}}(1),
\end{align*}
as $P_{m_n}\big(
    n^{-1}\sum_{i=1}^n |b_0(Z_i) - b_{m_n}(Z_i)| \ge Ma_n  
    \big) \le  \|b_0 - b_{m_n}\|_{L_1(P_Z)}/M a_n \le 1/M$. Finally, we can rewrite the third sum in~\eqref{eq:three_conditions} as $ n^{-1}\sum_{i=1}^n (W_i - b_{m_n}(Z_i)) \beta_{m_n}(Z_i)^\tr (\pi - \widehat{\pi})$, where $\pi = \E\,\{\beta_{m_n}(Z)W\}$ and $\widehat{\pi}$ is its empirical counterpart. Using, once more, Theorem 4.1 in \cite{belloni2015some} we have 
\begin{align*}
    \big|\meanin (W_i - b_{m_n}(Z_i)) \beta_{m_n}(Z_i)^\tr (\pi - \widehat{\pi})\big| &\le \xi_{m_n} \|\pi - \widehat{\pi}\|_{2} \meanin |W_i - b_{m_n}(Z_i) |\\
    &= O_{P_{m_n}}(\xi_{m_n}\sqrt{k_{m_n} / n}),
\end{align*}
as $P_{m_n}\big(
    n^{-1}\sum_{i=1}^n |W_i - b_{m_n}(Z_i) | \ge M
    \big) \le M^{-1} (\E\,|W| + \E\, |b_0(Z) | + a_n)$. Here we remark that the rate conditions in~\ref{plm-A-LAN-rates} can be relaxed if we impose additional conditions on $W$. In particular, if we replace~\ref{plm-A-LAN-rates} with 
\begin{enumerate}
    \item[(pl2$^\star$)] $W$ is sub-exponential with parameters $(\nu, \alpha)$,
     $\xi_{m_n}k_{m_n} / \sqrt{n}\to 0$ \& $a_{n}\sqrt{k_{m_n}}\xi_{m_n} \to 0$
\end{enumerate}
then the second and third terms can be shown to be $o_{P_{m_n}}(1)$ exactly as before but we can refine the argument relating to the first term. In particular,~\eqref{eq:plm-eq-belloni} holds as above, but the probabilistic bound on $n^{-1/2}\sum_{i=1}^n (W_i - b_{0}(Z_i))\beta_{m_n, j}(Z_i)$ can be improved. Since $b_0(Z_i)$ is bounded by, say $C_0$ and $\beta_{m_n, j}(Z_i)$ is bounded by $\xi_{m_n}$, $(W - b_{0}(Z))\beta_{m_n, j}(Z)$ is also subexponential with parameters $(\xi_{m_n}(\nu + C_0), \alpha)$. Hence for all $t\ge 0$ the Bernstein-type bound below holds (e.g.,~\citet[Eq.~(2.18), p.~29]{wainwright2019high})
\begin{equation*}
    P_{m_n}\left(\left|
        \meanin (W_i - b_{0}(Z_i))\beta_{m_n, j}(Z_i)\right|
    \ge t \right)  \le 2\exp\left(
    -\min\left\{\frac{nt}{2\alpha},     \frac{nt^2}{2\xi_{m_n}^2(\nu + C_0)^2}\right\} \right).
\end{equation*}
Hence taking $t = t_n = 2(\nu + C_0)\xi_{m_n} \sqrt{\log k_{m_n}} / \sqrt{n}$ we have that 
\begin{align*}
     &k_{m_n} P_{m_n}\left(\left|
        \meanin (W_i - b_{0}(Z_i))\beta_{m_n, j}(Z_i)\right|
    \ge t \right)
    \\  
    &\qquad\qquad\qquad \le k_{m_n}2\exp\left(-\frac{n n^{-1} 4\xi_{m_n}^2(\nu + C_0)^2 \log k_{m_n}}{2\xi_{m_n}^2(\nu + C_0)^2}\right)\\
    &\qquad\qquad\qquad= 2k_{m_n} \exp(-2 \log k_{m_n}) = 2/k_{m_n}\to0.
\end{align*}
Thus by~\eqref{eq:plm-eq-belloni} and the union bound
   \begin{equation*}
     \sum_{j=1}^{k_{m_n}} (\widehat{\gamma}_j(\theta_0) - \gamma_j) \fracrootn\sum_{i=1}^{n} \beta_{m_n, j}(Z_i) (W_i - b_0(Z_i)) = 
     O_{P_{m_n}}\big(\frac{\xi_{m_n}
     \sqrt{\log k_{m_n}} k_{m_n}}{n}\big)= 
     o_{P_{m_n}}(1).
\end{equation*}
Finally we provide an example of $\beta_{m_n}$ for which $h_{m_n, n} \to h$ in $L_2(P_0)$. In particular, let $\beta_{m_n}$ be (orthonormalised) B-splines of degree $l$ with equally spaced knots. Then if 
\begin{enumerate}[resume*=plmconditions, label = (pl\arabic*)]
    \item \label{plm-eta-smooth} $\eta$ is $s$-times continuously differentiable and $l+1\ge s$
    \item \label{plm-eta-rate} $k_{m_n}^{-s}\sqrt{n}\to 0$
\end{enumerate}
it follows from Theorem~20.3 in~\cite{powell1981approximation} that if $\gamma$ is chosen optimally then $\|\beta_{m_n}(z)^\tr \gamma  - \eta_0(z)\|_{\infty} =O(k_{m_n}^{-s})$ un turn implying that $\sqrt{n} \|\beta_{m_n}^\tr \gamma - \eta_0\|_{L_2(P_Z)} = O(\sqrt{n}k_{m_n}^{-s}) = o(1)$, and thus $h_{m_n, n} \to 0$ in $L_2(P_0)$.

\begin{remark} Locally constant basis functions are easy to work with, and provide very intelligible results. Consider therefore the sequence of orthonormal basis function $\beta_{m,j}$ that take the form $\beta_{m,j}(z) = I_{V_{m,j}}(z)/(\E\, I_{V_{m,j}}(Z) )^{1/2}$ for $j = 1,\ldots,k_m, \, k_m \geq 1$, where $V_{m,j} = \{z \colon (j-1)/k_m \leq z < j/k_m\}$ for $j = 1,\ldots,k_m$, and $V_{m,1}\cup \cdots \cup V_{m,k_m} = [0,1]$ for all $m$. Consider the following conditions:
 \begin{enumerate}[label=(lc\arabic*)]\itemsep-0.2em
 \item\label{cond:pl2} $Z$ has density $f_Z$ that is continuously differentiable and positive on $[0,1]$; 
 \item\label{cond:pl3} The function $z \mapsto \E\,(W^k \given Z = z)$ is continuously differentiable for $k=1,2$;
 \item\label{cond:pl5} $\eta_0 \in \mc{H}$, where $\mc{H}$ consists of all continuously differentiable functions $\eta \colon [0,1] \to \real$.
 \end{enumerate}
With these basis function, one may use Theorem~\ref{thm:get_Anh_LAN} to show that the quadratic expansion in~\eqref{eq:Anh_LAN} holds provided the subsequence $(m_n)$ is chosen so that 
$n/k_{m_n} \to \infty$; and that Assumption~\ref{ass:sieve-approx}, i.e., contiguity, is satisfied provided $\sqrt{n}/k_{m_n} \to 0$. 
\end{remark}

\subsection{The Cox model}\label{appendix:cox}
In this appendix we provide some details left out of Section~\ref{subsec:cox}. We need to show that $\log \dd P_{m_n}^n/\dd P_0^n = o_{P_0}(1)$; that the scores are indeed uniformly $P_{m_n}$-integrable; and that $A_{n}^{\rm cox}(h)$ admits a quadratic expansion, as claimed. Introduce 
\begin{equation}
S_n^{(k)}(t,\theta) = \frac{1}{n}\sum_{i=1}^n Y_i(t) W_i^k \exp(\theta W_i) ,\quad \text{for $k = 0,1,2$},
    \notag
\end{equation}
and define $s_{\theta,\eta}^{(k)}(t) = \E_{\theta,\eta}\,Y(t)W^k \exp(\theta W)$ for $k = 0,1,2,3$, so that due to the data being i.i.d., $s_{\theta,\eta}^{(k)}(t) = \E_{\theta,\eta}\,S_n^{(k)}(t,\theta)$. Write $s^{(k)} = s_{\theta_0,\eta_0}^{(k)}$ and $s_m^{(k)} = s_{\theta_0,T_m^{-1}\gamma_0}^{(k)}$, and define
\begin{equation}
v_{m,j} = 
\frac{\int_{V_{m,j}}s_m^{(2)}(s)\,\dd s}{\int_{V_{m,j}}s_m^{(0)}(s)\,\dd s}
- 
\bigg(\frac{\int_{V_{m,j}}s_m^{(1)}(s)\,\dd s}{\int_{V_{m,j}}s_m^{(0)}(s)\,\dd s}\bigg)^2,\quad \text{and}\quad 
v(t) = \frac{s^{(2)}(t)}{s^{(0)}(t)}
- \bigg(\frac{s^{(1)}(t)}{s^{(0)}(t)}\bigg)^2.
\notag
\end{equation}
We make the following assumptions:
\begin{enumerate}[label=(cx\arabic*)]\itemsep-0.2em
\item\label{cond:cx1} $\E\, W^k \exp(\theta W) < \infty$ for $k = 0,1,2,3$, for all $\theta$ in a neighbourhood of $\theta_0$ and $\E[|W|^{2+\delta}]<\infty$;
\item\label{cond:cx2} The baseline hazard $\eta_0$ is continuously differentiable and positive on $[0,1]$;
\item\label{cond:cx3} In a neighbourhood $U_{\theta_0}$ of $\theta_0$, 
\begin{equation}
\sup_{t \in [0,1] ,\theta \in U}|S_n^{(k)}(t,\theta) - s_{\theta,\eta_0}^{(k)}(t)| = o_{P_0}(1), \quad\text{for $k = 0,1,2$}.
\notag
\end{equation}
\item\label{cond:cx4} For a neighbourhood $U_{\eta_0}$ of $\eta_0$, the functions $(\theta,\eta)\mapsto s_{\theta,\eta}^{(k)}(t)$ for $k = 0,1,2,3$ are continuous on $U_{\eta_0}\times U_{\theta_0}$, uniformly in $t \in [0,1]$; they are bounded on $U_{\eta_0}\times U_{\theta_0}\times[0,1]$, and $s_{\theta,\eta}^{(0)}$ is bounded below on $U_{\eta_0}\times U_{\theta_0}\times[0,1]$, and $J = \int_0^1 v(t) s^{(0)}(t) \eta_0(t)\,\dd t$ is positive.
\end{enumerate}
Assumption~\ref{cond:cx2}, \ref{cond:cx3}, and \ref{cond:cx4} are similar to Assumption~A., B., and D.~of the canonical Cox regression paper \citet[p.~1105]{andersen1982cox}.

We start with the details involved in showing Assumption~\ref{ass:sieve-approx}. Let $(m_n)_{n \geq 1}$ be such that $\sqrt{n}/k_{m_n} \to 0$, and recall that 
\begin{equation} 
\log \frac{\dd P_{m_n}^n}{\dd P_0^n} = \frac{1}{\sqrt{n}}\sum_{i=1}^n \int_0^1 \frac{h_{m_n,n}}{\eta_0}\,\dd M_i
- \frac{1}{2n}\sum_{i=1}^n \int_0^1 \frac{h_{m_n,n}^2}{\eta_0^2}\,\dd N_i + r_{m_n,n},
\label{eq:cox_lr_ratio}
\end{equation}
with $h_{m,n}(t) = \sqrt{n}( \beta_m(t)^{\tr}\gamma_0 - \eta_0(t))$; $r_{m,n} = n^{-1}\sum_{i=1}^n\int_0^1 h_{m,n}^2/\eta_0^2R(h_{m,n}/(\sqrt{n}\eta_0))\,\dd N_i$; and $R$ is the function defined via the Taylor expansion $\log( 1 + a) = a - \half a^2 + a^2 R(a)$, in particular $R(a) = \half \big(1 - 1/(1 + \tilde{a})^2 \big)$, where $\tilde{a}$ is some point between $0$ and $a$. With the basis functions in~\eqref{eq:cox_piecewisebasis}, we have that $h_{m_n,n} \to 0$ in $L_2(\dd t)$ if and only if $(m_n)_{n \geq 1}$ is such that $\sqrt{n}/k_{m_n}\to 0$. But then 
\begin{equation}
\E_0\, \bigg( 
\frac{1}{\sqrt{n}}\sum_{i=1}^n \int_0^1 \frac{h_{m_n,n}(t)}{\eta_0(t)}\,\dd M_i(t)
\bigg)^2
=
\int_0^1 
\frac{h_{m_n,n}(t)^2}{\eta_0(t)^2} s^{(0)}(t) \eta_0(t) \,\dd t
= o(1), 
    \notag
\end{equation}
by the $L_2(\dd t)$ convergence $h_{m_n,n} \to 0$, combined with~\ref{cond:cx1} and~\ref{cond:cx4}. Similarly, the second term on the right in~\eqref{eq:cox_lr_ratio} is a nonnegative random variable with expectation
\begin{equation}
\E_0\, \frac{1}{n}\sum_{i=1}^n \int_0^1 \frac{h_{m_n,n}^2}{\eta_0^2}\,\dd N_i
= \int_0^1 \frac{h_{m_n,n}(t)^2}{\eta_0(t)^2} s^{(0)}(t) \eta_0(t)\,\dd t = o(1).
    \notag
\end{equation}
Thus both these terms tend to zero in probability by Markov{'}s inequality. To show that the remainder $r_{m_n,n}$ vanishes, note that since $\eta_0$ is continuous on $[0,1]$, hence uniformly continuous we can given $0 < \eps < 1$ find $n_0$ such that $\sup_{t\in [0,1]}|\beta_{m_n}(t)^{\tr}\gamma_0 - \eta_0(t)| < \eps$ for all $n \geq n_0$. But then, by the triangle inequality $|R(n^{-1/2} h_{m_n,n}(t)^2/\eta_0(t)^2 ) | \leq \half \big(1 + 1/(1 - \eps)\big)$ for all $n \geq n_0$. Consequently, $\E_0\, |r_{m_n,n}| \leq \half (1 + 1/(1 - \eps) ) \int_0^1 \big(h_{m_n,n}(t)^2/\eta_0(t)^2\big) s^{(0)}(t) \eta_0(t)\,\dd t$ for all $n \geq n_0$, and we have convergence to zero, and $r_{m_n,n} = o_{P_0}(1)$ by Markov{'}s inequality.  

It remains to verify~\eqref{eq:score-approx-parametric} and \eqref{eq:score-approx-nonparametric} of Lemma~\ref{lem:score-approx-implies-effscr-approx}, which we do using Lemma~\ref{lemma:effscr-approx-based-on-Pmn-TV}. To this end, consider the submodels $\tau \mapsto p_{\theta_0 + a \tau, \eta_0 + b \tau}$ where $a \in \real$ and $b \in B \subset \mc{H}$. Differentiating with respect to $\tau$, and evaluating in $\tau = 0$
\begin{equation}
\frac{\dd}{\dd \tau} \log p_{\theta_0 + a \tau, \eta_0 + b \tau}(X) \big|_{\tau = 0} = a \dot{\ell}_{\theta_0,\eta_0} + Db, 
\notag
\end{equation}
where $\dot{\ell}(X) = W M(1)$ and $Db = \int_0^1 b(s)/\eta_0(s)\,\dd M(s)$. Then
\begin{align*}
|\dot{\ell}_{m_n}(X) - \dot{\ell}(X) | 
 & \leq |W|\exp(\theta_0W ) 
\sum_{j=1}^{k_{m_n}}\int_{V_{m_n,j}}|  \eta_0(s) - \eta_0((j-1)/k_{m_n}) |\,\dd s\\
& \leq |W|\exp(\theta_0W )\half \sup_{t\in [0,1]}\eta_0^{\prime}(t)/k_{m_n},  
\end{align*}
which, as $k_{m_n}\to \infty$, tends to zero in probability by~\ref{cond:cx1}. 
Next, let $(\gamma_m)_{m \geq 1}$ be a sequence of vectors such that $\gamma_m^{\tr}\beta_m \to b$ uniformly on $[0, 1]$. We have the bound
\begin{align*}
& |\gamma_{m_n}^{\tr}\dot{v}_{m_n}(X) - Db(X) |\\
&\qquad \qquad \leq 
|\int_0^1 
\bigg( \frac{\gamma_{m_n}^{\tr}\beta_{m_n}(s)}{\gamma_0^{\tr}\beta_{m_n}(s)} - \frac{b(s)}{\eta_0(s)}        \bigg)\,\dd M^{m_n}(s)|
+\sup_{s\in [0,1]}\frac{b(s)}{\eta_0(s)} \exp(\theta_0W)/k_{m_n}.
\end{align*}
Here the first term on the right tends to zero in probability by an application of the It{\^o} isometry followed by Markov{'}s inequality (or, alternatively, Lenglart{'}s inequality). The second term tends to zero in probability by~\ref{cond:cx1} and \ref{cond:cx2}. 

In order to show that both $\dot{\ell}_{m_n}^2$ and $(\gamma_{m_n}^{\tr}\dot{v}_{m_n})^2$ are uniformly $P_{m_n}$-integrable, we may argue as follows. We start with  $\dot{\ell}_{m_n}^2$. Since its compensator is continuous, the quadratic variation of $M^{m_n}$ is $[M^{m_n}, M^{m_n}](t) = N(t)$.
  Let $Z^{m_n}\define WM^{m_n}(t)$ and note that $Z^{m_n}(1) = \dot{\ell}_{m_n}(X)$. 
  It is easy to check that $(Z^{m_n}(t))_{t\in[0, 1]}$ is a martingale.
  We also have that $Z^{m_n}(0) = W M^{m_n}(0) = W N(0) = 0$ a.s. since $T\ge 0$ and $T=0$ has zero probability. By the definition of quadratic variation, it is clear that $[Z^{m_n}, Z^{m_n}](t) = W^2N(t)$.
  Clearly $\dot{\ell}_{m_n}(X)^{2+\delta} \le \sup_{t\in [0, 1]} |Z^{m_n}(t)|^{2+\delta}$. Thus, using the Burkholder-Davis-Gundy inequality (e.g., \citet[Theorem~11.5.5, p.~249]{cohen2015stochastic}),
  \begin{equation*}
    \E_{m_n}\left[\sup_{t\in [0, 1]} |Z^{m_n}(t)|^{2+\delta}\right] \le C\, \E[W^{2+\delta}N(1)^{1+\delta/2}]\le C\,\E[W^{2+\delta}],
  \end{equation*}
  which is finite by Condition~\ref{cond:cx1}, and hence $\dot{\ell}_{m_n}^2$ is uniformly $P_{m_n}$-integrable.
  
  For $(\gamma_{m_n}^{\tr}\dot{v}_{m_n})^2$ we first note that by~\ref{cond:cx2} and the fact that $\gamma_{0}^\tr \beta_{m_n}\to \eta_0$ uniformly (as noted above), it follows straightforwardly that \begin{equation*}
    \liminf_{n\to\infty}\inf_{s\in [0, 1]} \gamma_{0}^\tr \beta_{m_n}(s) \ge c > 0,
  \end{equation*}
  for some $c$. Combining this with the uniform convergence of $\gamma_{m_n}^\tr\beta_{m_n}$ to $b$, it follows that for all large enough $n$, the ratio $\gamma_{m_n}^\tr\beta_{m_n}  /\gamma_{0}^\tr\beta_{m_n}$ is bounded.  Hence by Proposition II.4.1. in \cite{andersen1993statistical}, $K^{m_n}(t)\define \int_0^t (\gamma_{m_n}^\tr\beta_{m_n})/(\gamma_{0}^\tr\beta_{m_n})\darg{M^{m_n}}$ is a local square integrable martingale with $K^{m_n}(0)=0$ a.s., and 
  \begin{align*}
    \left[K^{m_n},\, K^{m_n}\right](t) 
     \le  \int_0^1 \left(\frac{\gamma_{m_n}^\tr\beta_{m_n}(s)}{\gamma_{0}^\tr\beta_{m_n}(s)}\right)^2\darg{N(s)}
    = \left(\frac{\gamma_{m_n}^\tr\beta_{m_n}(T)}{\gamma_{0}^\tr\beta_{m_n}(T)}\right)^2.
  \end{align*}
  For all large enough $n$, we have that $|\gamma_{m_n}^\tr\beta_{m_n}(s)| \le (\overline{b} + 1)$ where $\overline{b}\define \sup_{t\in [0, 1]} |b(t)|$ and also
  $\inf_{s\in[0, 1]}\gamma_{0}^\tr\beta_{m_n}(s) \ge c > 0$. Thus for such $n$ the right hand side above can be further bounded above by 
  \begin{equation*}
    \left(\frac{\gamma_{m_n}^\tr\beta_{m_n}(T)}{\gamma_{0}^\tr\beta_{m_n}(T)}\right)^2 \le \frac{\overline{b}+ 1}{c}.
  \end{equation*}
  Then by the Burkholder-Davis-Gundy inequality for a $\delta>0$, and all large enough $n$, 
  \begin{equation*}
    \E[(\gamma_{m_n}^\tr \dot{v}_{m_n})^{2+\delta}]\le \E\left[\sup_{t\in[0, 1]}K^{m_n}(t)^{2+\delta}\right]\le C \left(\frac{\overline{b} + 1}{c}\right)^{1+\delta/2},
  \end{equation*}  
which is finite, hence $(\gamma_{m_n}^{\tr}\dot{v}_{m_n})^2$ un uniformly $P_{m_n}$-integrable. 
\end{appendix}

\bibliographystyle{plainnat} 
\bibliography{refs_spcontig}

\end{document}